\documentclass[11pt]{amsart}
\usepackage{fullpage}
\usepackage{color}
\usepackage{graphicx,psfrag} 
\usepackage{color} 
\usepackage{tikz}
\usepackage{pgffor}
\usepackage{subfigure}

\usepackage{perpage}     

\MakePerPage{footnote}   

\newtheorem{theorem}{Theorem}[]
\newtheorem{lemma}[theorem]{Lemma}
\newtheorem{corollary}[theorem]{Corollary}
\newtheorem{proposition}[theorem]{Proposition}

\theoremstyle{definition}
\newtheorem{definition}[theorem]{Definition}
\newtheorem{example}[theorem]{Example}
\theoremstyle{remark}
\newtheorem{remark}[theorem]{Remark}

\newtheorem{property}[theorem]{Property}

\newtheorem{ex-prop}[theorem]{Example-Proposition}

\numberwithin{equation}{section}

\definecolor{gray}{rgb}{.5,.5,.5}

\definecolor{black}{rgb}{0,0,0}

\definecolor{blue}{rgb}{0,0,1}

\definecolor{red}{rgb}{1,0,0}
\def\red{\color{red}}
\definecolor{green}{rgb}{0,1,0}

\definecolor{yellow}{rgb}{1,1,.4}

\begin{document}

\title{Links arising from braid monodromy factorizations}

\author{Meirav Amram}
\address{Department of Mathematics, Bar-Ilan University, Ramat-Gan, 52900, Israel\newline and Shamoon College of Engineering,
Bialik/Basel Sts., Beer-Sheva 84100, Israel}

\author{Moshe Cohen}
\address{Department of Mathematics, Bar-Ilan University, Ramat Gan 52900, Israel}

\author{Mina Teicher}
\address{Department of Mathematics, Bar-Ilan University, Ramat Gan 52900, Israel}





\begin{abstract}
We investigate the local contribution of the braid monodromy factorization in the context of the links obtained
by the closure of these braids.  We consider plane curves which are arrangements of lines and conics as well as
some algebraic surfaces, where some of the former occur as local configurations in degenerated and regenerated
surfaces in the latter.  In particular we focus on degenerations which involve intersection points of
multiplicity two and three.  We demonstrate when the same links arise even when the local arrangements are
different.
\end{abstract}

\maketitle

\section{Introduction}
\label{sec:Intro}

Braid monodromy is a tool for studying topologically all kinds of  curve configurations on $2$-dimensional complex surfaces, or branch curves, and also of plane curves which do not appear as branch curves.  These computations fit into a program started by Moishezon-Teicher \cite{MT:simp, MTI, MTII, MTIII, MTIV, MTV} on using braid monodromy factorization as invariants of connected components of the moduli space of surfaces of general type.

In this paper we consider both plane curves and also branch curves of algebraic surfaces as the ramifications of the projection in order to get equivalences between their monodromies and the closures of the monodromies.

In the case of plane curves, there was much work done considering line arrangements \cite{GTV:wiring} and conic-line arrangements \cite{AGT:2conics, AGT:6, ATU}.

In the case of branch curves, the work of Moishezon-Teicher motivates the classification of algebraic surfaces by developing invariants  which differentiate between the components of the moduli space.

This work is followed by the degenerations and regenerations of surfaces and then heavy enumerations of objects in the degenerated surfaces and computations of monodromy and related groups.  Degenerations and regenerations of surfaces are tools to make these calculations more manageable; they are used in, for example, \cite{AGTV, ARST, AO, ATV:Hirz, ATV:TxT, CMT}.

In order to reduce these technical steps we find in this paper equivalences between various degenerations via the resulting monodromies.  We use a courser invariant theory by taking their closures, translating the problem to knot and link theory, through which we can say the braids are distinct when the links obtained are distinct.


The overall objective in this context is to understand the building blocks of the braids that appear in the
braid monodromy factorization. This leads to a deeper understanding of factorizations that can possibly occur
which in turn may shed light on the type of surfaces arising.

In the case of a plane curve, we construct the table of monodromy of Moishezon-Teicher and consider the closures of braids obtained in this way.  In the case of a branch curve, we embed an algebraic surface $X$ in a projective space, and we study its degeneration into a union of planes $X_0$. Ordering the edges and vertices in $X_0$ lexicographically \cite{MT:simp} and in other ways, we get multiple intersection points which are known from \cite{ACMT, MT:simp, MTIV} as $k$-points, where $k$ is the number of edges meeting at the vertex.  In this paper we focus on 2- and 3-points.  Projecting $X_0$ onto $\mathbb{CP}^2$, we get a line arrangement $S_0$ and using the regeneration rules
\cite{MTIV}, we recover $S$ the branch curve of $X$, used to compute the braid monodromy factorization.  See for example \cite{ACMT, AFT:PxT, AFT:2Hirz, AT:TxT2}.  

As the braid monodromy technique gives us the braid monodromy factorization of $S$, it is an invariant which distinguishes between connected components of the moduli space.
We note that when we consider the affine part of the curves and compute the monodromy, we do not get the braid monodromy factorization but its part without the braids of the singularities at infinity. 

From the perspective of knots, we investigate the braids in the braid monodromy factorization by way of local intersection points in the algebraic surfaces and categorize its building blocks.  We plan to give a complete list for higher multiple intersection points and
expect to obtain new local orderings in deformations of surfaces to answer various questions in algebraic geometry.



This paper contributes to the eventual classification of algebraic surfaces.  We implement the closures of the braids to make it easier to understand some known examples.  By studying these as links, more simplification can be done.  We consider link invariants such as the number of components and the linking numbers between these components; we simplify further by reducing cables until the links we obtain appear on the Knot Atlas \cite{knotatlas} as prime.
We demonstrate that the same building blocks appear in several of these examples, and this local information will be used to build larger degenerated surfaces.  The degeneration pictures (seen in Figures \ref{fig:surface}, \ref{fig:surface2pts}, \ref{fig:pillow}, \ref{fig:example:CP1xCP1}, and  \ref{fig:surface2pts}) contain only the affine part, and this is why we ignore the singularities at infinity.  These similarities are difficult to recognize at the level of the braids, especially before the computations needed to simplify them.


The paper is organized as follows.  In Section \ref{sec:Background} we recall plane curves and the braid monodromy that can be obtained from the different singularity types.  We give alternate notations for braids, knots, and links.  In Section \ref{sec:Line} we compute monodromies and their closures that are related to plane curves coming from line arrangements and conic-line arrangements, and we find equivalences between some configurations in Propositions \ref{prop:conicline} and \ref{prop:conicconic}.  In Section \ref{sec:surfaces} we introduce the notion of degeneration and consider the different types of $2$- and $3$-points that arise under different enumerations of the vertices and lines in the degenerations.  In Section \ref{sec:Main} we compute the related monodromies and their closures for all types of $2$- and $3$-points.  In addition to these local contributions we also consider global symmetries in Proposition \ref{thm:mirror}.  In Section \ref{sec:examples} we give interesting examples of regenerations of 2-points  and 3-points and the results concerning their braids and closures. We show that the closures of monodromies of all types of 2-points are the same. This holds also for one type of 3-points:  in the regeneration process we get two double lines and one conic. The second type of 3-points is the exceptional case:  in the regeneration process we get two conics and one double line, giving only one possible labelling.

\section{Background}
\label{sec:Background}

We will follow the braid monodromy algorithm of Moishezon-Teicher \cite{MTI, MTII}.  A detailed treatment may also be found in \cite{ACMT, AT:TxT1}.

Algebraically a braid is a word in the Artin group $B_m$ generated by $\sigma_1,\ldots,\sigma_{m-1}$, where
geometrically $\sigma_i$ takes the $i$-th strand over the $(i+1)$-st strand and acts as the identity on the others, as in Figure \ref{fig:sigma}.  
See also \cite{Art}.

\begin{figure}[h]
\begin{center}
\begin{tikzpicture}

\draw[dotted] (4.5,-1.5) ellipse (5cm and 1cm);

\draw[line width=5pt, color=white] (5,0) .. controls (5,-.5) and (4,-1) .. (4,-1.5);
\draw (5,0) .. controls (5,-.5) and (4,-1) .. (4,-1.5);

\draw[line width=5pt, color=white] (4,0) .. controls (4,-.5) and (5,-1) .. (5,-1.5);
\draw (4,0) .. controls (4,-.5) and (5,-1) .. (5,-1.5);

\draw[line width=2pt, color=white] (4.5,0) ellipse (5cm and 1cm);
\draw[dotted] (4.5,0) ellipse (5cm and 1cm);

\foreach \x in {2,4,5,7}
    {   \fill[color=black] (\x,0) circle (3pt); }
\draw (3,0) node {$\ldots$};
\draw (6,0) node {$\ldots$};

\draw (2,.5) node {$1$};
\draw (4,.5) node {$i$};
\draw (5,.5) node {$i+1$};
\draw (7,.5) node {$m$};

\foreach \x in {2,4,5,7}
    {   \fill[color=black] (\x,-1.5) circle (3pt);  }
\draw (3,-1.5) node {$\ldots$};
\draw (6,-1.5) node {$\ldots$};

\draw (2,0) -- (2,-1.5);
\draw (7,0) -- (7,-1.5);

\end{tikzpicture}
    \caption{The element $\sigma_i$.}
    \label{fig:sigma}
\end{center}
\end{figure}
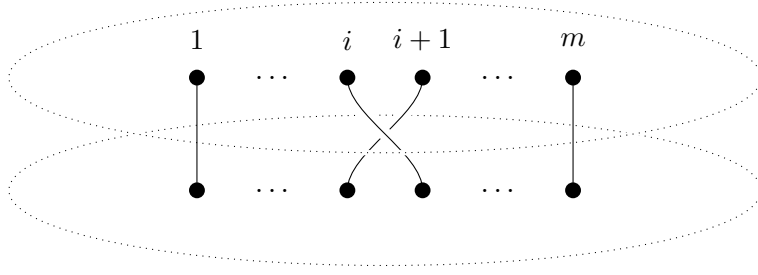

The element $\sigma_i$ can also be denoted $Z_{i\; i+1}$.  
This generalizes to $Z_{i\; j}$ as in Figure \ref{fig:skeleton}.  The reader unfamiliar with the $Z_{i\; j}$ notation may choose to read \cite{MTI}.

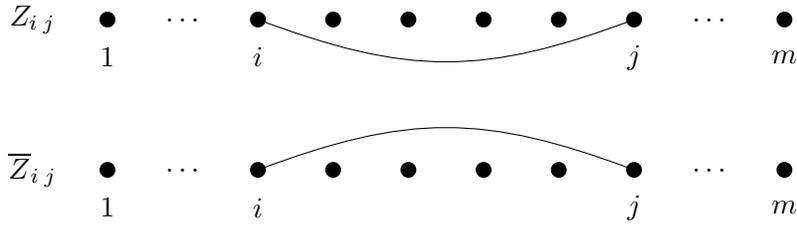
\begin{figure}[h]
\begin{center}
\begin{tikzpicture}

\draw (-1,0) node {$Z_{i\; j}$};

\fill[color=black] (0,0) circle (3pt);
\draw (0,-.5) node {$1$};
\draw (1,0) node {$\ldots$};
\fill[color=black] (2,0) circle (3pt);
\draw (2,-.5) node {$i$};
\fill[color=black] (3,0) circle (3pt);
\fill[color=black] (4,0) circle (3pt);
\fill[color=black] (5,0) circle (3pt);
\fill[color=black] (6,0) circle (3pt);
\draw (7,-.5) node {$j$};
\fill[color=black] (7,0) circle (3pt);
\draw (8,0) node {$\ldots$};
\fill[color=black] (9,0) circle (3pt);
\draw (9,-.5) node {$m$};

\draw (2,0) .. controls (4,-.75) and (5,-.75) .. (7,0);

\draw (-1,-2) node {$\overline{Z}_{i\; j}$};

\fill[color=black] (0,-2) circle (3pt);
\draw (0,-2.5) node {$1$};
\draw (1,-2) node {$\ldots$};
\fill[color=black] (2,-2) circle (3pt);
\draw (2,-2.5) node {$i$};
\fill[color=black] (3,-2) circle (3pt);
\fill[color=black] (4,-2) circle (3pt);
\fill[color=black] (5,-2) circle (3pt);
\fill[color=black] (6,-2) circle (3pt);
\draw (7,-2.5) node {$j$};
\fill[color=black] (7,-2) circle (3pt);
\draw (8,-2) node {$\ldots$};
\fill[color=black] (9,-2) circle (3pt);
\draw (9,-2.5) node {$m$};

\draw (2,-2) .. controls (4,-1.25) and (5,-1.25) .. (7,-2);

\end{tikzpicture}
    \caption{The braids associated with $Z_{i\; j}$ and $\overline{Z}_{i\; j}$.}
    \label{fig:skeleton}
\end{center}
\end{figure}


\begin{property}
\label{property:node}
The braids $Z_{i\;j}$ and $\overline{Z}_{i\;j}$ can be re-written as
$$Z_{i\;j}=(\sigma_i\ldots\sigma_{j-2})\sigma_{j-1}(\sigma_i\ldots\sigma_{j-2})^{-1}=(\sigma_{i+1}\ldots\sigma_{j-1})^{-1}\sigma_{i}(\sigma_{i+1}\ldots\sigma_{j-1}),$$
$$\overline{Z}_{i\;j}=(\sigma_{j-2}\ldots\sigma_i)^{-1}\sigma_{j-1}(\sigma_{j-2}\ldots\sigma_i)=(\sigma_{j-1}\ldots\sigma_{i+1})\sigma_{i}(\sigma_{j-1}\ldots\sigma_{i+1})^{-1}.$$
\end{property}


\begin{property}
\label{property:highermult}
The braid $Z_{i\;i+1\;\ldots\; i+k}$ can be re-written as
$$Z_{i\;i+1\;\ldots\; i+k}=(\sigma_{i+k-1}\ldots\sigma_{i})(\sigma_{i+k-1}\ldots\sigma_{i+1})\ldots(\sigma_{i+k-1}\sigma_{i+k-2})(\sigma_{i+k-1}).$$
\end{property}

For a more in-depth overview of knots and links, including linking number, see for example \cite{KT}.

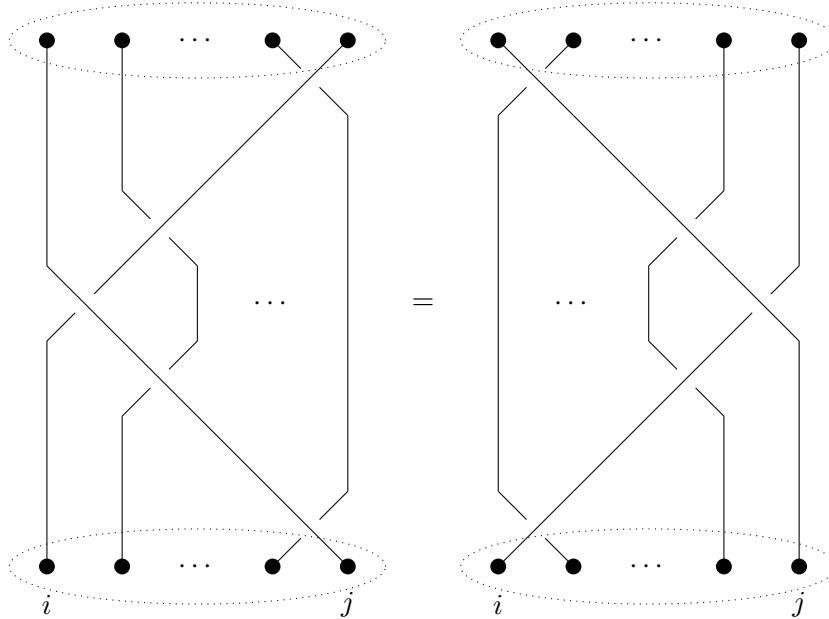
\begin{figure}[h]
\begin{center}
\begin{tikzpicture}

\foreach \x/ \y in {0/3,1/2,3/0}
    {
    \draw (\x,\y) -- (\x+1,\y+1);
    \draw[color=white, line width=10] (\x+1,\y) -- (\x,\y+1);
    \draw (\x+1,\y) -- (\x,\y+1);
    }

\foreach \x/ \y in {1/4, 3/6}
    {
    \draw (\x+1,\y) -- (\x,\y+1);
    \draw[color=white, line width=10] (\x,\y) -- (\x+1,\y+1);
    \draw (\x,\y) -- (\x+1,\y+1);
    }

\draw[dotted] (2,0) ellipse (2.5cm and .5cm);
\draw[dotted] (2,7) ellipse (2.5cm and .5cm);

\draw (0,0) -- (0,3);
\draw (0,4) -- (0,7);
\draw (1,0) -- (1,2);
\draw (1,5) -- (1,7);
\draw (2,2) -- (3,1);
\draw (2,3) -- (2,4);
\draw (2,5) -- (3,6);
\draw (4,1) -- (4,6);

\draw (2,0) node {$\cdots$};
\draw (3,3.5) node {$\cdots$};
\draw (2,7) node {$\cdots$};

\foreach \x in {0,1,3,4}
    {   \fill[color=black] (\x,0) circle (3pt); }
\foreach \x in {0,1,3,4}
    {   \fill[color=black] (\x,7) circle (3pt); }
\draw (0,-.5) node {$i$};
\draw (4,-.5) node {$j$};

\draw (5,3.5) node {$=$};

\foreach \x/ \y in {9/3,8/4,6/6}
    {
    \draw (\x,\y) -- (\x+1,\y+1);
    \draw[color=white, line width=10] (\x+1,\y) -- (\x,\y+1);
    \draw (\x+1,\y) -- (\x,\y+1);
    }

\foreach \x/ \y in {8/2, 6/0}
    {
    \draw (\x+1,\y) -- (\x,\y+1);
    \draw[color=white, line width=10] (\x,\y) -- (\x+1,\y+1);
    \draw (\x,\y) -- (\x+1,\y+1);
    }

\draw[dotted] (8,0) ellipse (2.5cm and .5cm);
\draw[dotted] (8,7) ellipse (2.5cm and .5cm);

\draw (10,0) -- (10,3);
\draw (10,4) -- (10,7);
\draw (9,0) -- (9,2);
\draw (9,5) -- (9,7);
\draw (8,2) -- (7,1);
\draw (8,3) -- (8,4);
\draw (8,5) -- (7,6);
\draw (6,1) -- (6,6);

\draw (8,0) node {$\cdots$};
\draw (7,3.5) node {$\cdots$};
\draw (8,7) node {$\cdots$};

\foreach \x in {6,7,9,10}
    {   \fill[color=black] (\x,0) circle (3pt); }
\foreach \x in {6,7,9,10}
    {   \fill[color=black] (\x,7) circle (3pt); }
\draw (6,-.5) node {$i$};
\draw (10,-.5) node {$j$};

\end{tikzpicture}
    \caption{The related braid $Z_{i\; j}$.}
    \label{fig:conjugatedsigma}
\end{center}
\end{figure}

Braids given by conjugation in the $Z_{i\; j}$ notation can be much more complicated.  Figure \ref{fig:skeleton2} gives a step-by-step method to determine, for example, the relatively easy conjugation $(Z_{3\; 8})^{Z^2_{3\; 5}Z^2_{3\; 4}}$.

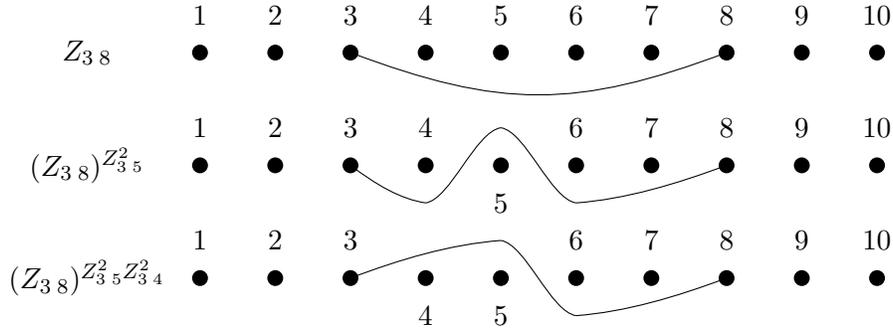
\begin{figure}[h]
\begin{center}
\begin{tikzpicture}

\draw (-1.5,0) node {$Z_{3\; 8}$};

\foreach \x in {1,...,10}
    {
    \fill[color=black] (\x-1,0) circle (3pt);
    \draw (\x-1,.5) node{\x};
    }

\draw (2,0) .. controls (4,-.75) and (5,-.75) .. (7,0);

\draw (-1.5,-1.5) node {$(Z_{3\; 8})^{Z^2_{3\; 5}}$};

\foreach \x in {1,2,3,4,6,7,8,9,10}
    {
    \fill[color=black] (\x-1,-1.5) circle (3pt);
    \draw (\x-1,-1) node{\x};
    }
    \fill[color=black] (4,-1.5) circle (3pt);
    \draw (4,-2) node{5};

\draw (5,-2) .. controls (5.66,-1.95) and (6.33,-1.75) .. (7,-1.5);

\draw (2,-1.5) .. controls (2.33,-1.75) and (2.66,-1.95) .. (3,-2);

\draw (4,-1) .. controls (4.33,-1.05) and (4.66,-1.95) .. (5,-2);

\draw (3,-2) .. controls (3.33,-1.95) and (3.66,-1.05) .. (4,-1);


\draw (-1.5,-3) node {$(Z_{3\; 8})^{Z^2_{3\; 5}Z^2_{3\; 4}}$};

\foreach \x in {1,2,3,6,7,8,9,10}
    {
    \fill[color=black] (\x-1,-3) circle (3pt);
    \draw (\x-1,-2.5) node{\x};
    }
\fill[color=black] (3,-3) circle (3pt);
    \draw (3,-3.5) node{4};
\fill[color=black] (4,-3) circle (3pt);
    \draw (4,-3.5) node{5};

\draw (2,-3) .. controls (2.66,-2.75) and (3.33,-2.55) .. (4,-2.5);

\draw (5,-3.5) .. controls (5.66,-3.45) and (6.33,-3.25) .. (7,-3);

\draw (4,-2.5) .. controls (4.33,-2.55) and (4.66,-3.45) .. (5,-3.5);


\end{tikzpicture}
    \caption{An example of a conjugated braid $(Z_{3\; 8})^{Z^2_{3\; 5}Z^2_{3\; 4}}$.}
    \label{fig:skeleton2}
\end{center}
\end{figure}

\subsection{Singularity types}

We detail the different types of singularities in a plane curve.

\subsubsection{Branch point}  For convenience here we consider only conics opening to the right as in Figure \ref{fig:conic}.  Note that to the right of the singularity, each typical fiber has two real intersection points.  However, to the left of the singularity the intersection points are complex.

The associated exponent $\varepsilon$ of a branch point is 1.

\begin{figure}[h]
\begin{center}
\begin{tikzpicture}

\draw (3,0) .. controls (0,1) and (0,2) .. (3,3);
\fill[color=black] (.75,1.5) circle (3pt);
\draw (-.5,1.5) node {branch point};


\end{tikzpicture}
    \caption{The conic and its branch point.}
    \label{fig:conic}
\end{center}
\end{figure}
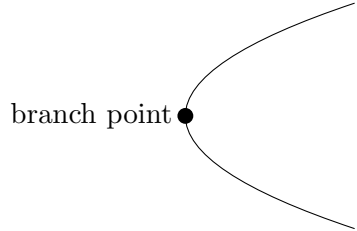


\subsubsection{Node}  A node is the intersection of two components.  The associated exponent $\varepsilon$ of a node is 2.  

\begin{example}
\label{exampleNode}
Consider the arrangement of three lines forming a triangle as in Figure \ref{fig:triangle}.

\begin{figure}[h]
\begin{center}
\begin{tikzpicture}

\draw (0,1) -- (4,1);
\draw (0,0) -- (3,3);
\draw (1,3) -- (4,0);

\draw (4.5,0) node {$1$};
\draw (4.5,1) node {$2$};
\draw (4.5,3) node {$3$};

    \fill[color=black] (1,1) circle (3pt);
\draw (1.25,.75) node {$3$};
    \fill[color=black] (2,2) circle (3pt);
\draw (2,1.6) node {$2$};
    \fill[color=black] (3,1) circle (3pt);
\draw (2.75,.75) node {$1$};

\end{tikzpicture}
    \caption{An arrangement of three lines.}
    \label{fig:triangle}
\end{center}
\end{figure}
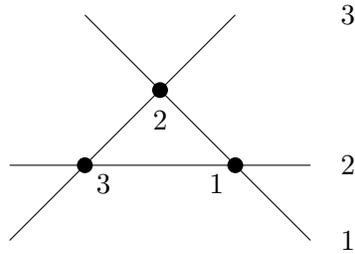




The braids related to the three nodes are $Z^2_{1\;2}$, $\overline{Z}^2_{1\;3}$, and $Z^2_{2\;3}$.  Their product can be expressed as:
\begin{equation*}
\sigma_1^2\cdot(\sigma^{-1}_1\sigma_2\sigma_1)^2\cdot\sigma_2^2=\sigma_1\sigma_2^2\sigma_1\sigma_2^2.
\end{equation*}

The closure 
of this braid is the torus link $T(3,3)$.
\end{example}

\subsubsection{Intersection point at infinity}
\label{subsec:NodeInf}  The associated exponent $\varepsilon$ of an intersection point at infinity is 2.  
  For several parallel lines, consider the lexicographic ordering of the related braids when multiplying them together (considering the ordering from the right-hand side).

\subsubsection{Tangency}  The associated exponent $\varepsilon$ of a tangency is 4.  

In Section \ref{sec:surfaces}, we explain in more detail that the regeneration of a tangency yields three cusps, each of whose associated exponent $\varepsilon$ is 3.  See \cite{MTIV}.


\subsubsection{Intersection points with higher multiplicities $k>2$}  The associated exponent $\varepsilon$ of an intersection point with multiplicity $k>2$ is 2.
  This exchanges the $i$-th and $(i+(k-1))$-st positions, the ($i+1$)-st and $(i+(k-1)-1)$-st positions, and so on.

\section{Plane arrangements with lines and conics}
\label{sec:Line}

In this section, we consider three kinds of arrangements:  first those with only lines, then those with only conics, and then those with both.

First we note that all local arrangements of $m$ lines give the same local contribution to the braid
monodromy factorization.  In particular, the closure of this braid is the well-studied torus link on
$m$ strands twisted $m$ times.

In this section we begin with generic line arrangements, where each pair of lines intersects at exactly one distinct node.  From there we consider parallel lines as well as central arrangements, in order to produce the more general case, for example the one seen in Example \ref{ex:general}.


Given a generic arrangement of $m$ lines, every pair of lines intersects exactly once, and every intersection point is a node where exactly two lines meet.  Then the $\binom{m}{2}$ nodes, each with degree 2, contribute to the total degree $m(m-1)$ of the braid monodromy factorization of $\Delta^2$, which is two full twists on $m$ strands.

\begin{proposition} (Moishezon-Teicher) \cite[Proposition-Example VIII.2.1]{MTI}
For a generic arrangement of $m$ lines, that is, with only double points and with no parallel lines, the braid monodromy factorization is equivalent to (with respect to an equivalent relation that has not been defined in this paper) 
$$\prod_{\ell=2}^m\prod_{k=1}^{\ell-1}Z^2_{k\; \ell}.$$
\end{proposition}

Observe that this is written in some reverse-lexicographic ordering not corresponding to the order of the singularities on the axis.


\begin{property}
\label{prop:linearr}
Consider the closure of $\Delta^2$ on $m$ strands.  This link has the following properties:
\begin{enumerate}
    \item it has $m$ components,
    \item each of the $m$ components is itself unknotted, and
    \item each pair of components considered alone is the Hopf link.
\end{enumerate}
Furthermore it is the torus link $T(m,m)$ on $m$ strands twisted $m$ times.
\end{property}

One can also think of the lines on the plane as great circles on the sphere, and with the condition that each node becomes a positive crossing, this achieves the link described above.

Consider next the central arrangement of $m$ lines.  Then one need only consider the intersection point with multiplicity $m$ as described above, and the braid monodromy factorization has $\Delta^2$ as its one factor, which was treated in the generic case above.  See Property \ref{property:highermult}.

Now consider the arrangement of $m$ parallel lines.  Then the ``intersection point at infinity'' in the lexicographic order is as described above, and again the braid monodromy factorization has $\Delta^2$ as its one factor.  However, this point at infinity will not be considered in the local contribution to the braid monodromy factorization.
  See the following example.

\begin{example}
\label{ex:general}
Consider the line arrangement in Figure \ref{fig:exampleGeneral} with a triple point and parallel lines.

\begin{figure}[h]
\begin{center}
\begin{tikzpicture}

\draw (0,1) -- (4,1);
\draw (0,0) -- (3,3);
\draw (1,3) -- (4,0);

\draw (0,2) -- (2,0);

\draw (2.5,0) node {$1$};

\draw (4.5,0) node {$2$};
\draw (4.5,1) node {$3$};
\draw (3.5,3) node {$4$};

    \fill[color=black] (1,1) circle (3pt);
\draw (1,1.5) node {$3$};
    \fill[color=black] (2,2) circle (3pt);
\draw (2,1.5) node {$2$};
    \fill[color=black] (3,1) circle (3pt);
\draw (3,1.5) node {$1$};

\end{tikzpicture}
    \caption{A general example of a line arrangement.}
    \label{fig:exampleGeneral}
\end{center}
\end{figure}
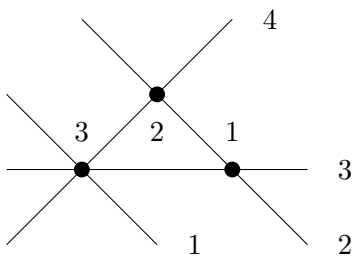

The braid monodromy table is given in Table \ref{tab:exampleGeneral}.

\begin{table}[h]
    \centering
    \caption{The braid monodromy table for a more general arrangement of four lines.}
\vspace{.25cm}
\begin{tabular}{c|c|c}
singularity &   exponent $\varepsilon$   & related braid\\
\hline
 &  & \\
1   &   2 &  $Z^2_{2\;3}$ \\
 &  & \\
2   &   2 &  $\overline{Z}^2_{2\;4}$ \\
 &  & \\
3   &   2 &  $Z^2_{1\;3\;4}$
\end{tabular}
\vspace{.2cm}
    \label{tab:exampleGeneral}
\end{table}

The product of the five resulting braids is given in Equation (\ref{eqn:example:fivelines}):
\begin{equation}
\label{eqn:example:fivelines}
\sigma_2^2\cdot(\sigma_2^{-1}\sigma_3^2\sigma_2)\cdot[\sigma_1(\sigma_3\sigma_2\sigma_3)^2\sigma_1^{-1}]
=\sigma_2\sigma_3\sigma_1\sigma_2(\sigma_1\sigma_2\sigma_3)^2.
\end{equation}

A horizontal depiction of this braid is given Figure \ref{fig:ex:fourlines}.

\begin{figure}[h]
\begin{center}
\begin{tikzpicture}

\foreach \x/ \y in {0/1, 1/0, 1/2, 2/1, 3/2, 4/1, 5/0, 5/2, 6/1, 7/0}
    {
    \draw (\x+1,\y) -- (\x,\y+1);
    \draw[color=white, line width=10] (\x,\y) -- (\x+1,\y+1);
    \draw (\x,\y) -- (\x+1,\y+1);
    }


\draw (0,0) -- (1,0);
\draw (2,0) -- (5,0);
\draw (6,0) -- (7,0);
\draw (3,1) -- (4,1);
\draw (7,2) -- (8,2);
\draw (0,3) -- (1,3);
\draw (2,3) -- (3,3);
\draw (4,3) -- (5,3);
\draw (6,3) -- (8,3);

\end{tikzpicture}
    \caption{The braid obtained from Example \ref{ex:general}.}
    \label{fig:ex:fourlines}
\end{center}
\end{figure}
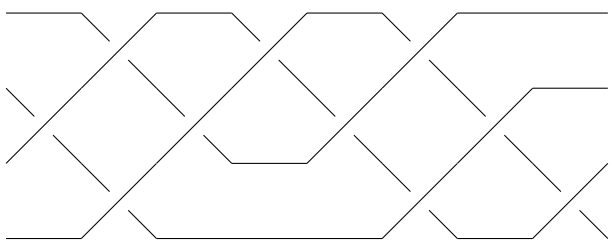

The closure of this braid gives a link of four components, each of which is itself an unknot with no crossings.  The intersection point at infinity, which is not included here in this local contribution, would link the first two components.  All other pairs have linking number one.

This link can also be interpreted as the Hopf link L2a1 on the Thistlethwaite Link Table on the Knot Atlas \cite{knotatlas} whose first component unknot is replaced by its (2,0)-cable and whose second component unknot is replaced by its (2,2)-cable.  Recall that a ($p,t$)-cable has $p$ parallel copies and $t$ twists.
\end{example}

Next we consider some arrangements of conics and lines together. In particular, some of these appear in the
local partial regenerations of surfaces as in Section \ref{sec:surfaces}.

\begin{proposition}
\label{prop:conicline} Consider the two arrangements of a single conic and a single line shown in Figure
\ref{fig:example:conicline}.

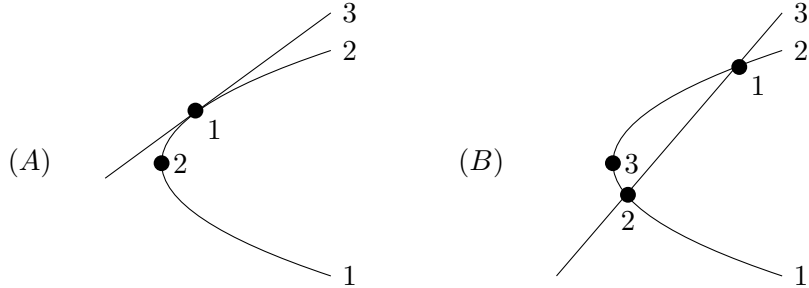
\begin{figure}[h]
\begin{center}
\begin{tikzpicture}

\draw (-1,1.5) node {$(A)$};

\draw (3,0) .. controls (0,1) and (0,2) .. (3,3);
\draw (3,3.5) -- (0,1.3);
\fill[color=black] (.75,1.5) circle (3pt);
\draw (1,1.5) node {$2$};
\fill[color=black] (1.2,2.2) circle (3pt);
\draw (1.45,1.95) node {$1$};

\draw (3.25,0) node {$1$};
\draw (3.25,3) node {$2$};
\draw (3.25,3.5) node {$3$};

\draw (5,1.5) node {$(B)$};

\draw (9,0) .. controls (6,1) and (6,2) .. (9,3);
\draw (9,3.5) -- (6,0);
\fill[color=black] (6.75,1.5) circle (3pt);
\draw (7,1.5) node {$3$};
\fill[color=black] (6.95,1.08) circle (3pt);
\draw (6.95,.73) node {$2$};
\fill[color=black] (8.43,2.78) circle (3pt);
\draw (8.68,2.53) node {$1$};

\draw (9.25,0) node {$1$};
\draw (9.25,3) node {$2$};
\draw (9.25,3.5) node {$3$};

\end{tikzpicture}
    \caption{Two arrangements of a single conic and a single line.}
    \label{fig:example:conicline}
\end{center}
\end{figure}

Then the braids given by the local contribution of the braid monodromy factorization are the same.
Furthermore, the link obtained by the closure of this braid is 
the torus link $T(2,4)$.  This is also the alternating link L4a1 on the Thistlethwaite Link Table on the Knot Atlas \cite{knotatlas} as shown in Figure \ref{fig-L4a1}.
\end{proposition}

%

\begin{figure}[h]
  \begin{center}
    \includegraphics[scale=0.65]{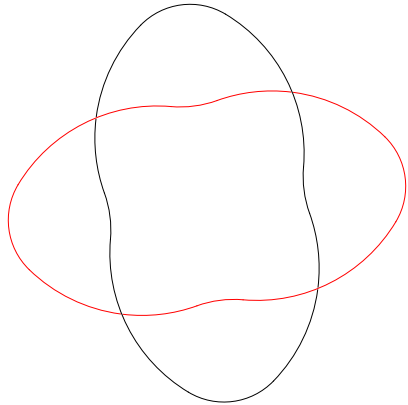}
  \end{center}

  \caption{The alternating link L4a1.}
  \label{fig-L4a1}
\end{figure}

\begin{proof}
Consider the arrangement of a single line tangent to a single conic as in Figure \ref{fig:example:conicline} (A).
The braid monodromy table for just these two points gives $Z_{2\;3}^4$ and $(Z_{1\;2})^{Z^2_{2\;3}}$,
contributing locally $\sigma_2^4\cdot(\sigma_2^{-2}\sigma_1\sigma_2^2)=\sigma_2^2\sigma_1\sigma_2^2$.

Alternatively one could slide the line of this arrangement so that it passes through the conic twice as in
Figure \ref{fig:example:conicline} (B), creating two nodes and the branch point.  Here the resulting braids are
$Z^2_{2\;3}$, $Z^2_{1\;3}$, and $Z_{1\;2}$, which contribute the product
$\sigma_2^2\cdot(\sigma_2^{-1}\sigma_1^2\sigma_2)\cdot\sigma_1=\sigma_2\sigma_1^2\sigma_2\sigma_1=
\sigma_2^2\sigma_1\sigma_2^2$  like above.


After a stabilization move, the braid is $\sigma_2^4$, whose closure is $T(2,4)$.
\end{proof}

Next we consider some arrangements of conics only. Observe that Figure \ref{fig:example:conicconic} (B) appears in the partial regeneration of the 3-point with the second type as considered in Example \ref{ex:3ptTypeII} and Figure \ref{fig:two3ptcases} (II).

\begin{proposition}
\label{prop:conicconic} Consider the two arrangements of two conics shown in Figure
\ref{fig:example:conicconic}.

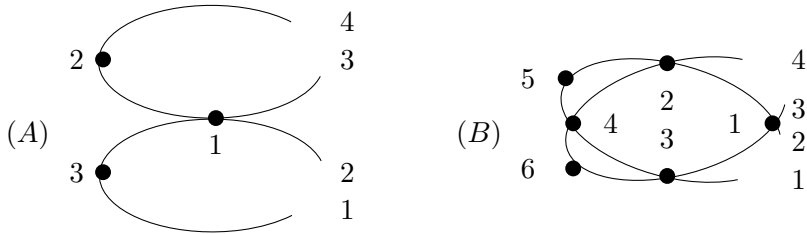
\begin{figure}[h]
\begin{center}
\begin{tikzpicture}

\draw (-1,1.5) node {$(B)$};

\draw[shift={(-.5,1)}] [rotate around={15:(3,1.5)}] (3,1.5) arc (45:345:1.5 and .75);
\draw[shift={(0,0)}] [rotate around={-15:(3,1.5)}] (3,1.5) arc (15:315:1.5 and .75);

\fill[color=black] (2.9,1.65) circle (3pt);
\draw (2.4,1.65) node {$1$};
\fill[color=black] (1.5,2.45) circle (3pt);
\draw (1.5,1.95) node {$2$};
\fill[color=black] (1.5,.95) circle (3pt);
\draw (1.5,1.45) node {$3$};
\fill[color=black] (.25,1.65) circle (3pt);
\draw (.75,1.65) node {$4$};
\fill[color=black] (.15,2.25) circle (3pt);
\draw (-.35,2.25) node {$5$};
\fill[color=black] (.25,1.05) circle (3pt);
\draw (-.35,1.05) node {$6$};

\draw (3.25,.9) node {$1$};
\draw (3.25,1.4) node {$2$};
\draw (3.25,1.85) node {$3$};
\draw (3.25,2.5) node {$4$};

\draw (-7,1.5) node {$(A)$};

\draw[shift={(-6.5,1.5)}] (3,1.5) arc (45:345:1.5 and .75);
\draw[shift={(-6.1,-.35)}] (3,1.5) arc (15:315:1.5 and .75);

\fill[color=black] (-4.5,1.72) circle (3pt);
\draw (-4.5,1.37) node {$1$};
\fill[color=black] (-6,1) circle (3pt);
\draw (-6.35,1) node {$3$};
\fill[color=black] (-6,2.5) circle (3pt);
\draw (-6.35,2.5) node {$2$};

\draw (-2.75,.5) node {$1$};
\draw (-2.75,1) node {$2$};
\draw (-2.75,2.5) node {$3$};
\draw (-2.75,3) node {$4$};

%
%
%
%

\end{tikzpicture}
    \caption{Two conic arrangements.}
    \label{fig:example:conicconic}
\end{center}
\end{figure}

Then the first link obtained by the closure of the braids given in the local contribution to the braid monodromy
factorization is L4a1 as on \cite{knotatlas} as shown in Figure \ref{fig-L4a1}, and the second is L4a1 with one component replaced with its (2,1)-cable.
\end{proposition}

%
%

%

\begin{proof}
Consider the arrangement of two conics with a single tangency along with the two branch points as in Figure \ref{fig:example:conicconic} (A).  The braid monodromy table for these three points gives the related braids $Z_{2\;3}^4$, $(Z_{3\;4})^{Z^2_{2\;3}}$, and $(Z_{1\;2})^{Z^2_{2\;3}}$ contributing locally $\sigma_2^4\cdot(\sigma_2^{-2}\sigma_3\sigma_2^2)\cdot(\sigma_2^{-2}\sigma_1\sigma_2^2)=\sigma_2^2\sigma_3\sigma_1\sigma_2^2$.  After two stabilization moves on the closure, this gives the same as the closure of $\sigma_2^4$ on two strands, which is the four-crossing alternating link L4a1 on \cite{knotatlas} as shown in Figure \ref{fig-L4a1}.

Alternatively consider the arrangement of two conics with no tangencies, giving four nodes along with the two branch points as in Figure \ref{fig:example:conicconic} (B).  The braid monodromy table for these six points gives the related braids $Z_{2\;3}^2$, $\overline{Z}_{2\;4}^2$, $Z^2_{1\;3}$, $(Z_{1\;4}^2)^{Z_{3\;4}^{-2}}$, $Z_{3\;4}$, and $Z_{1\;2}$.  Together these contribute
\begin{equation*} \sigma_2^2\cdot(\sigma_3\sigma_2^2\sigma_3^{-1})\cdot(\sigma_2^{-1}\sigma_1^2\sigma_2)\cdot(\sigma_3\sigma_2^{-1}\sigma_1^2\sigma_2\sigma_3^{-1})\cdot\sigma_3\cdot\sigma_1
\end{equation*}
\begin{equation}
\label{eqn:example:twoconicsnodes}
=\sigma_2\sigma_3^2\sigma_1^2\sigma_2\sigma_3\sigma_1\sigma^2_2,
\end{equation}
whose closure is the same as $\sigma_2\sigma_1^3\sigma_2\sigma_1^3\sigma_2$ on three strands after a stabilization move and some conjugation.  This gives a two-component link where each component interacts as in L4a1 on the Knot Atlas \cite{knotatlas} but where one component unknot is replaced with its (2,1)-cable.  Recall that a ($p,t$)-cable has $p$ parallel copies and $t$ twists.

This is also the two-component link where each component interacts as in the Hopf link L2a1 on the Knot Atlas \cite{knotatlas} but where each component replaced with its (2,1)-cable.
\end{proof}

\begin{remark}
The arrangements which appear in Figures  \ref{fig:example:conicline} and \ref{fig:example:conicconic} are the
affine pieces of the projective arrangements. The projective arrangements are not themselves interesting because they yield
the complete braid monodromy factorization, which always gives the torus link $T(m,m)$.
\end{remark}

\section{Surfaces}
\label{sec:surfaces}

Let $X$ be an algebraic surface embedded in projective space $\mathbb{CP}^n$.
Projecting $X$ onto the projective plane $\mathbb{CP}^2$ we get its
branch curve $S$, that is,
the ramification locus of the projection.

In order to better understand the complicated branch curve $S$,
we degenerate the surface $X$.  A partial degeneration gives a union of squares, each of which is homeomorphic to $\mathbb{CP}^1\times \mathbb{CP}^1$, using horizontal and vertical lines.  Each square is further degenerated into two planes by adding diagonal lines to obtain a union $X_0$ of triangles representing planes.  A detailed explanation of the degeneration process may be found in \cite{MTIII} and in further work including \cite{ACMT, AG, AO, AT:TxT1}, for example.  We take the following definition from \cite{ACMT}.

\begin{definition}
\label{df1}
Let $D$ be the unit disc,
and $X, Y$ be algebraic surfaces (or more generally algebraic varieties).
Suppose that $k: Y \rightarrow \mathbb{CP}^n$ and $k': X \rightarrow \mathbb{CP}^n$
are projective embeddings.
We say that $k'$ is a \emph{projective degeneration} of $k$
if there exist a flat family $\pi: V \rightarrow D$,
and an embedding $F:V\rightarrow D \times \mathbb{CP}^n$,
such that $F$ composed with the first projection is $\pi$,
and:
\begin{itemize}
\item[(a)] $\pi^{-1}(0) \simeq X$;
\item[(b)] there is a $t_0 \neq 0$ in $D$ such that
$\pi^{-1}(t_0) \simeq Y$;
\item[(c)] the family $V-\pi^{-1}(0) \rightarrow D-{0}$
is smooth;
\item[(d)] restricted to $\pi^{-1}(0)$, $F = {0}\times k'$
under the identification of $\pi^{-1}(0)$ with $X$;
\item[(e)] restricted to $\pi^{-1}(t_0)$, $F = {t_0}\times k$
under the identification of $\pi^{-1}(t_0)$ with $Y$.
\end{itemize}
\end{definition}

To demonstrate we explain this process for the degeneration of the Hirzebruch surface $F_2(1,2)$ into six
planes (the triangles) in Figure \ref{fig:surface}.  There are five lines separating the six planes once we ignore the boundary.  The intersections of these five lines occur at six points, four of which has multiplicity two and so are called 2-points (as marked by black vertices in the figure) and two of which are called 1-points (as marked by white vertices).  We project $X_0$ onto the plane and obtain a branch curve $S_0$ depicted by these five lines with these six intersection points.  In order to recover the original branch curve $S$ of $X$, we must regenerate $S_0$.


%

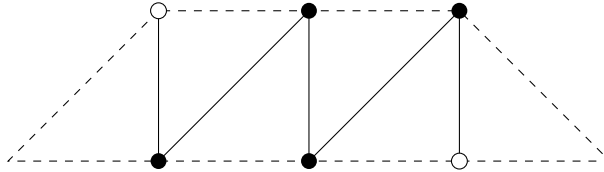
\begin{figure}[h]
\begin{center}
\begin{tikzpicture}

\draw (2,2) -- (2,0) -- (4,2) -- (4,0) -- (6,2) -- (6,0);
\draw[dashed] (2,2) -- (0,0) -- (8,0) -- (6,2) -- cycle;
\fill[color=black] (2,0) circle (3pt);
\fill[color=black] (4,0) circle (3pt);
\fill[color=black] (4,2) circle (3pt);
\fill[color=black] (6,2) circle (3pt);

	\fill[color=white] (2,2) circle (3pt);
	\draw (2,2) circle (3pt);
	\fill[color=white] (6,0) circle (3pt);
	\draw (6,0) circle (3pt);

\end{tikzpicture}
     \caption{The Hirzebruch surface $F_2(1,2)$ degenerated into six planes.}
     \label{fig:surface}
\end{center}
\end{figure}

  Now we explain in general the regeneration process, for any branch
curve $S_0$. The degenerated branch curve
$S_0$ has degree say $m$. However each of the $m$ lines of $S_0$
should be counted as a double line in the
scheme-theoretic branch locus, since it arises from a line of nodes.
Another way to see this is to note that the
regeneration of $X_0$ induces a regeneration of $S_0$ in such a way
that each point, say $c$, on the typical
fiber is replaced by two nearby points $c, c'$. The resulting branch
curve $S$ will have degree $2m$.

In full generality the branch curve $S_0$ has $k$-points as intersections for any $k$.  The regeneration process for large $k$ can be quite difficult, but work has been done for some values:  see \cite{FT:5}, \cite{AT:TxT1}, and \cite{AGST:8} for 5-, 6-, and 8-points, respectively.  In this work we restrict our attention to regenerations of just $2$- and $3$-points.

We show in Subsections \ref{subsec:41} and \ref{subsec:42} that different types of these $k$-points arise based on different orderings of the lines.  These line orderings are determined by orderings on the vertices.  In additional to the common lexicographic ordering, others are given.

The lines intersecting at these $k$-points are horizontal, vertical, or diagonal based on the construction from squares and triangles.  In the regeneration process we do not differentiate between horizontal and vertical lines; these will eventually regenerate into double lines.  On the other hand, diagonal lines will regenerate into conics.  We are concerned with intersections of these lines and conics.  The following lemma from \cite{MTIV} explains how tangencies arise in the regeneration process.

\begin{lemma}\label{lm:24}(Moishezon-Teicher)\cite{MTIV}
Let $V$ be a projective algebraic surface, and $C'$ be a curve in $V$.
Let $f: V \rightarrow \mathbb{C}
\mathbb{P}^2$ be a generic projection. Let $S \subseteq \mathbb{C}
\mathbb{P}^2, S' \subset V$ be the branch
curve of $f$ and the corresponding ramification curve. Assume $S'$
intersects $C'$ at a point $\alpha'$. Let $C
= f(C')$ and $\alpha = f(\alpha')$. Assume that there exist
neighborhoods of $\alpha $ and $\alpha '$, such that
$f_{\mid_{S'}}$ and $f_{\mid_{C'}}$ are isomorphisms. Then $C$ is
tangent to $S$ at $\alpha $.
\end{lemma}

We start with a regeneration of $2$-points. A diagonal line in a $2$-point is regenerated to a conic which is tangent to the horizontal or vertical line, as in Lemma \ref{lm:24} and Figure \ref{fig:regen2pt}.  In the partial regeneration we get a curve with a tangency and a branch point.  In order to complete the regeneration process, the line is doubled and the tangency is regenerated to three cusps.

\begin{figure}[h]
\begin{center}
\begin{tikzpicture}

\draw (-3,1) -- (-2,2) -- (-2,1);
\fill[color=black] (-2,2) circle (3pt);

\draw (-1,1.5) node {$\rightarrow$};

\draw (3,0) .. controls (0,1) and (0,2) .. (3,3);
\draw (3,3.5) -- (0,1.3);
\fill[color=black] (.75,1.5) circle (3pt);
\fill[color=black] (1.2,2.2) circle (3pt);


\end{tikzpicture}
    \caption{Partial regeneration in a neighborhood of a 2-point.}
    \label{fig:regen2pt}
\end{center}
\end{figure}
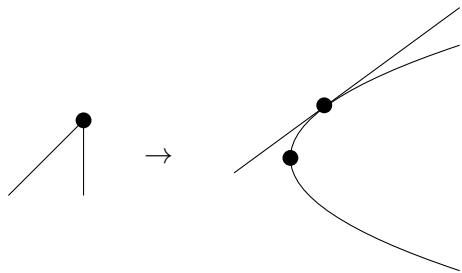

Now we study the regeneration of $3$-points.  The three intersecting lines can be horizontal, vertical, or diagonal, and one might suspect that any configuration is possible.  Due to the construction by squares, however, there can be no $3$-point with no diagonals and no $3$-point with all diagonals.  Thus there are two types of $3$-points, as in Figure \ref{fig:two3ptcases}.

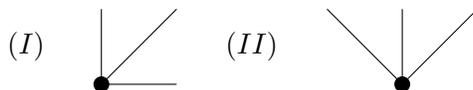
\begin{figure}[h]
\begin{center}
\begin{tikzpicture}

\draw (-1,0.5) node {$(I)$};

\draw (0,1) -- (0,0) -- (1,0);
\draw (0,0) -- (1,1);
\fill[color=black] (0,0) circle (3pt);

\draw (2,0.5) node {$(II)$};

\draw (3,1) -- (4,0) -- (5,1);
\draw (4,0) -- (4,1);
\fill[color=black] (4,0) circle (3pt);

\end{tikzpicture}
    \caption{The two types of a $3$-point.}
    \label{fig:two3ptcases}
\end{center}
\end{figure}

 For the first type, the regeneration is divided into steps. We explain each step in two
levels, first dealing with the surface and then with the branch curve.
At the surface level, each diagonal is replaced with a conic by a partial regeneration. Focusing on a
$3$-point, we have a partial regeneration of two of the planes to a quadric surface. We get one quadric and one plane,
which is tangent to the quadric. The plane and the quadric meet along two lines (one from each ruling of the
quadric). At the branch curve level, we have two double lines (coming from the intersection of the plane and the
quadric) and one conic (coming from the branching of the quadric over the plane). According to Lemma
\ref{lm:24}, the conic is tangent to each of the two double lines.
As far as the branch points go, one of the two branch points of the
conic is far away from the
$3$-point, and the other one is close to the $3$-point. See Figure \ref{fig:fig15}.

\begin{figure}[h]
\begin{center}
\begin{tikzpicture}

\draw (-8,1) -- (-9,1) -- (-9,2);
\draw (-9,1) -- (-8,2);
\fill[color=black] (-9,1) circle (3pt);

\draw (-7,1.5) node {$\rightarrow$};

\draw[rotate around={15:(-3,3)}] (-3,3) arc (45:345:1.5 and .75);
\draw (-2.25,2.2) -- (-6,0);
\draw (-2.25,.75) -- (-3,.95) -- (-6,1.75);
\fill[color=black] (-5.35,1.7) circle (3pt);
\fill[color=black] (-4.9,1.45) circle (3pt);
\fill[color=black] (-3.95,1.2) circle (3pt);
\fill[color=black] (-3,1.75) circle (3pt);

\fill[color=white] (-3.5,3) rectangle (-2,4);

\end{tikzpicture}
    \caption{Partial regeneration in a neighborhood of the $3$-point of the first type.}
    \label{fig:fig15}
\end{center}
\end{figure}
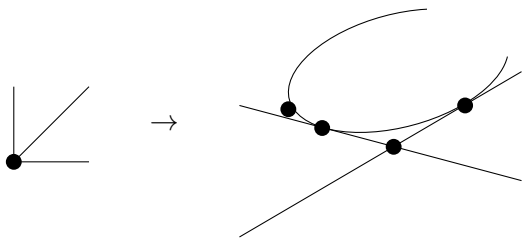


In the next step of the regeneration, at the curve level, we use regeneration lemmas from \cite{MTII}. The two tangent points regenerate to three cusps each (giving a total of six) and the intersection point of the two
double lines gives eight more branch points. One can think of this as
first giving four nodes, then each node
giving two branch points. At the surface level it means that we get a
smooth surface which locally looks like a
cubic in $\mathbb{C}\mathbb{P}^3$ (degenerating to a triple of planes).

Now we study the regeneration of the second type of a $3$-point. One
of the diagonal lines regenerates to a
conic which is tangent to the vertical line. The vertical line and the
second diagonal line (which is still yet to be regenerated)  intersect at a $2$-point. In the second step of the
regeneration, the second diagonal line is
regenerated to a conic which is tangent to the vertical line, too.  See Figure \ref{fig:fig16}.

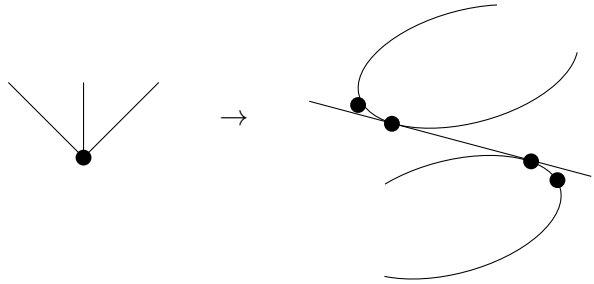
\begin{figure}[h]
\begin{center}
\begin{tikzpicture}

\draw (-10,2) -- (-9,1) -- (-8,2);
\draw (-9,1) -- (-9,2);
\fill[color=black] (-9,1) circle (3pt);

\draw (-7,1.5) node {$\rightarrow$};

\draw[rotate around={15:(-3,3)}] (-3,3) arc (45:345:1.5 and .75);
\draw (-2.25,.75) -- (-3,.95) -- (-6,1.75);
\fill[color=black] (-5.35,1.7) circle (3pt);
\fill[color=black] (-4.9,1.45) circle (3pt);

\fill[color=white] (-3.5,3) rectangle (-2,4);

\draw[rotate around={15:(-5,-.58)}] (-5,-.58) arc (-135:165:1.5 and .75);
\fill[color=black] (-3.05,.95) circle (3pt);
\fill[color=black] (-2.7,.7) circle (3pt);
\fill[color=white] (-6,0) rectangle (-5,.75);

\end{tikzpicture}
    \caption{Partial regeneration in a neighborhood of the $3$-point of the second type.}
    \label{fig:fig16}
\end{center}
\end{figure}

In the last step, the vertical line is doubled and each tangency regenerates to three cusps. Note that the two conics intersect in four complex points (not shown).

Now we formulate the regenerated braid for a node and for a tangency using the regeneration rules from \cite{MTIV}.



Recall that the braid $Z^2_{ij}$ is a full-twist of $j$ around $i$ and the braid $Z^2_{i'j}$ is a full-twist of $j$ around $i'$. The braid $Z^2_{ii',j}$ is obtained by a regeneration (the point $i$ on the typical fiber is replaced by $i, i'$), and it is a full-twist of $j$ around $i$ and $i'$.

\begin{theorem}\label{2rule}
{\bf Regeneration rule for a node} (Moishezon-Teicher) \cite[p. 337]{MTIV}\\
A factor of the form $Z^2_{ij}$ regenerates to $Z^2_{ii',j},
Z^2_{i,jj'}$ or $Z^2_{ii',jj'}$.
\end{theorem}

In the last case, the regeneration of a node can be depicted as in Figure \ref{fig:noderegeneration}.

\begin{figure}[h]
\begin{center}
\begin{tikzpicture}

\draw (-.5,0) -- (1.5,2);
\draw (1.5,0) -- (-.5,2);
\fill[color=black] (.5,1) circle (3pt);
\draw (2,0) node {$i$};
\draw (2,2) node {$j$};

\draw (3,1) node {$\rightarrow$};

\draw (4,.5) -- (6,2.5);
\draw (4,-.5) -- (6,1.5);
\draw (4,1.5) -- (6,-.5);
\draw (4,2.5) -- (6,.5);
\fill[color=black] (4.5,1) circle (3pt);
\fill[color=black] (5,1.5) circle (3pt);
\fill[color=black] (5,.5) circle (3pt);
\fill[color=black] (5.5,1) circle (3pt);
\draw (6.5,-.5) node {$i$};
\draw (6.5,.5) node {$i'$};
\draw (6.5,1.5) node {$j$};
\draw (6.5,2.5) node {$j'$};


\end{tikzpicture}
    \caption{Regeneration of a node.}
    \label{fig:noderegeneration}
\end{center}
\end{figure}
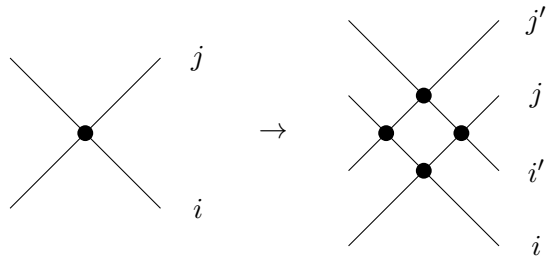

These three cases are explicitly given by (a), (b), and (f) of the following lemma:

\begin{lemma}
\label{lem:25} (Amram-Teicher) \cite[Lemma 10]{AT:TxT1} The following hold:
(a) $Z^2_{ii',j} = Z^2_{i'j} Z^2_{ij}$; \\
(b) $Z^{2}_{i',jj'} = Z^{2}_{i'j'} Z^{2}_{i'j}$;
(c) $Z^{-2}_{i',jj'}  = Z^{-2}_{i'j} Z^{-2} _{i'j'}$;
(d) $\bar{Z}^{-2}_{i',jj'} = \bar{Z}^{-2}_{i'j'} \bar{Z}^{-2} _{i'j}$;
(e) $Z^{-2}_{ii',j} = Z^{-2}_{ij} Z^{-2} _{i'j}$; \\
(f) $Z^2_{ii',jj'} = Z^2_{i',jj'} Z^2_{i,jj'}$;
(g) $Z^{-2}_{ii',jj'} = Z^{-2}_{i,jj'} Z^{-2}_{i',jj'}$.
\end{lemma}

Recall that a tangency is regenerated into three cusps.

\begin{theorem}\label{3rule}
{\bf Regeneration rule for a tangency} (Moishezon-Teicher) \cite[p. 337]{MTIV}\\
A factor of the form $Z^4_{ij}$ regenerates to $Z^{3}_{i,jj'} = (Z^3_{ij})^{Z_{jj'}} \cdot (Z^3_{ij}) \cdot
(Z^3_{ij})^{{Z_{jj'}}^{-1}}$ or to \\
$Z^{3}_{ii',j}=(Z^3_{ij})^{Z_{ii'}}  \cdot (Z^3_{ij}) \cdot (Z^3_{ij})^{{Z_{ii'}}^{-1}}$.
\end{theorem}

Next we consider $2$- or $3$-points that appear in degenerations of surfaces.

\subsection{Degeneration and regeneration of $2$-points}
\label{subsec:41}

$2$-points appear in degenerations as intersections of two lines. We can find them in degenerations of Hirzebruch
surfaces, in a self-product of the projective line, in a product of the projective line with a complex torus, or
in some toric varieties. In  Figure \ref{fig:surface2pts}, for example, we show a degeneration of the surface
$\mathbb{CP}^1 \times \mathbb{CP}^1$, embedded by the bi-linear system (1,2). The black vertices in the figure
are $2$-points.

\begin{figure}[h]
\begin{center}
\begin{tikzpicture}

\draw (0,0) -- (4,0) -- (4,2) -- (0,2) -- cycle;
\draw (0,0) -- (2,2) -- (2,0) -- (4,2);



\fill[color=black] (2,0) circle (3pt);
\fill[color=black] (2,2) circle (3pt);


\end{tikzpicture}
    \caption{The $2$-points in the surface $\mathbb{CP}^1 \times \mathbb{CP}^1$.}
    \label{fig:surface2pts}
\end{center}
\end{figure}
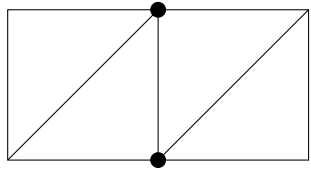

The ordering on the degenerated surface is determined by an ordering on the vertices.  There are several ways of
doing this, but the most convenient one is the lexicographic ordering following work by Moishezon and Teicher.
  Consider vertices $v_1=(x_1,y_1)$ and $v_2=(x_2,y_2)$. Then $v_1<v_2$ if
and only if $y_1<y_2$ or $y_1=y_2$ and $x_1<x_2$.  That is, enumerate vertices starting from the lower left corner proceeding to the right and then continuing upwards.
  Consider lines $L_1$ through vertices $u_1$ and $v_1$ and
$L_2$ through vertices $u_2$ and $v_2$ under the condition that $u_1<v_1$ and $u_2<v_2$.  Then $L_1<L_2$ if and
only if $v_1<v_2$ or $v_1=v_2$ and $u_1<u_2$. For example, the vertices and the lines of Figure \ref{fig:surface2pts}
can be enumerated as in Figure \ref{fig:surface2ptsnumbered}.

\begin{figure}[h]
\begin{center}
\begin{tikzpicture}

\draw (0,0) -- (4,0) -- (4,2) -- (0,2) -- cycle;
 \draw (0,0) -- (2,2) -- (2,0) -- (4,2);
 \draw (0,-.5) node {$v_1$};
 \draw (2,-.5) node {$v_2$};
  \draw (4,-.5) node {$v_3$};
   \draw (0,2.5) node {$v_4$};
    \draw (2,2.5) node {$v_5$};
\draw (4,2.5) node {$v_6$};

 \draw (1.3,1) node {$e_1$};
 \draw (2.25,1) node {$e_2$};
 \draw (3.3,1) node {$e_3$};

\fill[color=black] (2,0) circle (3pt);
\fill[color=black] (2,2) circle (3pt);

\end{tikzpicture}
    \caption{Enumerations of lines and vertices.}
    \label{fig:surface2ptsnumbered}
\end{center}
\end{figure}
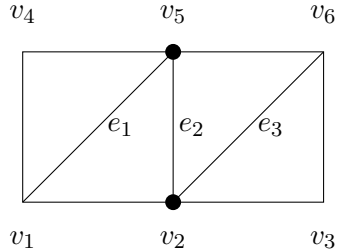

Choose a $2$-point arbitrarily. By the regeneration lemmas in \cite{MTII}, the diagonal line is regenerated  to a
conic which is tangent to the line as in Figure \ref{fig:2ptcasesAlgGeom}.  This figure shows the two possibilities of partial regeneration, depending on the ordering of the components.

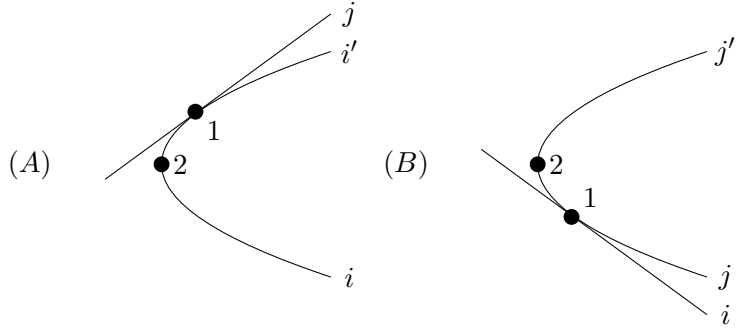
\begin{figure}[h]
\begin{center}
\begin{tikzpicture}

\draw (-1,1.5) node {$(A)$};

\draw (3,0) .. controls (0,1) and (0,2) .. (3,3);
\draw (3,3.5) -- (0,1.3);
\fill[color=black] (.75,1.5) circle (3pt);
\draw (1,1.5) node {$2$};
\fill[color=black] (1.2,2.2) circle (3pt);
\draw (1.45,1.95) node {$1$};

\draw (3.25,0) node {$i$};
\draw (3.25,3) node {$i'$};
\draw (3.25,3.5) node {$j$};

\draw (4,1.5) node {$(B)$};

\draw (8,0) .. controls (5,1) and (5,2) .. (8,3);
\draw (8,-.5) -- (5,1.7);
\fill[color=black] (5.75,1.5) circle (3pt);
\draw (6,1.5) node {$2$};
\fill[color=black] (6.2,.8) circle (3pt);
\draw (6.45,1.05) node {$1$};

\draw (8.25,-.5) node {$i$};
\draw (8.25,0) node {$j$};
\draw (8.25,3) node {$j'$};

\end{tikzpicture}
    \caption{Two cases of partial degeneration of a $2$-point.}
    \label{fig:2ptcasesAlgGeom}
\end{center}
\end{figure}

\begin{remark}
Observe that the possibility in Figure \ref{fig:2ptcasesAlgGeom} (A) occurs as the vertex labelled $v_5$ in Figure \ref{fig:surface2ptsnumbered} and the possibility in Figure \ref{fig:2ptcasesAlgGeom} (B) occurs as the vertex labelled $v_2$.
\end{remark}

In order to complete the regeneration process, each tangent line is regenerated to two parallel lines, and the tangency is replaced by three cusps (following the regeneration rules of \cite{MTIV}).

By the Moishezon-Teicher table of monodromy, we can compute the braids which are related to the singularities in Figure \ref{fig:2ptcasesAlgGeom},
and then by the above rule we get a curve of degree $4$ with three cusps and a branch point of a conic.
Figure \ref{fig:2ptcases} (A1) and (B1) are related to the three cusps which we get from the tangency in Figure \ref{fig:2ptcasesAlgGeom} (A) and (B), respectively.  Figure \ref{fig:2ptcases} (A2) and (B2) are related to the branch point of the conic in Figure \ref{fig:2ptcasesAlgGeom} (A) and (B), respectively.

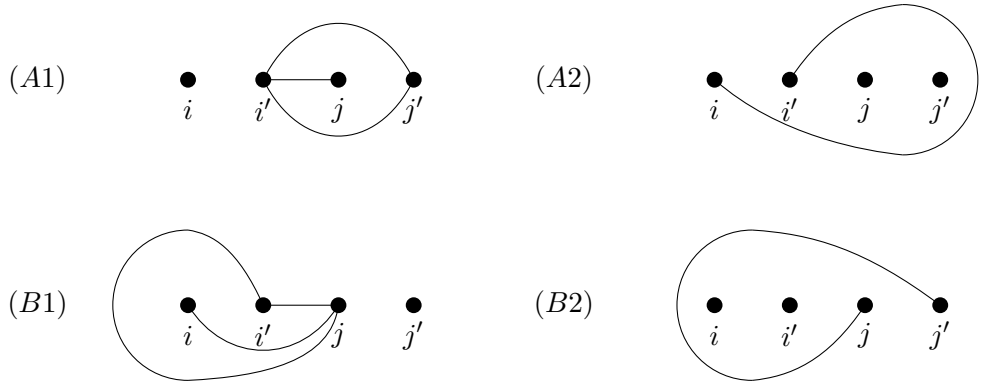
\begin{figure}[h]
\begin{center}
\begin{tikzpicture}

\draw (-2,0) node {$(A1)$};

\fill[color=black] (0,0) circle (3pt);
\fill[color=black] (1,0) circle (3pt);
\fill[color=black] (2,0) circle (3pt);
\fill[color=black] (3,0) circle (3pt);
\draw (0,-0.4) node {$i$};
\draw (1,-0.4) node {$i'$};
\draw (2,-0.4) node {$j$};
\draw (3,-0.4) node {$j'$};

\draw (1,0) -- (2,0);
\draw (1,0) .. controls (1.5,1) and (2.5,1) .. (3,0);
\draw (1,0) .. controls (1.5,-1) and (2.5,- 1) .. (3,0);

\draw (5,0) node {$(A2)$};

\fill[color=black] (7,0) circle (3pt);
\fill[color=black] (8,0) circle (3pt);
\fill[color=black] (9,0) circle (3pt);
\fill[color=black] (10,0) circle (3pt);
\draw (7,-0.4) node {$i$};
\draw (8,-0.4) node {$i'$};
\draw (9,-0.4) node {$j$};
\draw (10,-0.4) node {$j'$};

\draw (9.5,-1) arc (-90:90:1);
\draw (8,0) .. controls (8.5,.75) and (9,.95) .. (9.5,1);
\draw (7,0) .. controls (7.85,-.75) and (9,-.95) .. (9.5,-1);

\draw (-2,-3) node {$(B1)$};

\fill[color=black] (0,-3) circle (3pt);
\fill[color=black] (1,-3) circle (3pt);
\fill[color=black] (2,-3) circle (3pt);
\fill[color=black] (3,-3) circle (3pt);
\draw (0,-3.4) node {$i$};
\draw (1,-3.4) node {$i'$};
\draw (2,-3.4) node {$j$};
\draw (3,-3.4) node {$j'$};

\draw (1,-3) -- (2,-3);
\draw (0,-2) arc (90:270:1);
\draw (0,-2) .. controls (.33,-2.05) and (.66,-2.25) .. (1,-3);
\draw (0,-4) .. controls (1,-3.95) and (1.9,-3.75) .. (2,-3);
\draw (0,-3) .. controls (0.5,-3.8) and (1.5,-3.8) .. (2,-3);

\draw (5,-3) node {$(B2)$};

\fill[color=black] (7,-3) circle (3pt);
\fill[color=black] (8,-3) circle (3pt);
\fill[color=black] (9,-3) circle (3pt);
\fill[color=black] (10,-3) circle (3pt);
\draw (7,-3.4) node {$i$};
\draw (8,-3.4) node {$i'$};
\draw (9,-3.4) node {$j$};
\draw (10,-3.4) node {$j'$};

\draw (7.5,-2) arc (90:270:1);
\draw (7.5,-2) .. controls (8.3,-2.05) and (9,-2.25) .. (10,-3);
\draw (7.5,-4) .. controls (8,-3.95) and (8.5,-3.75) .. (9,-3);


\end{tikzpicture}
    \caption{Braids related to $2$-point cases with $j=i+1$.}
    \label{fig:2ptcases}
\end{center}
\end{figure}

\subsection{Degeneration and regeneration of $3$-points}
\label{subsec:42}

A $3$-point in a degeneration is an intersection of three lines.  We consider two types as in Figure \ref{fig:two3ptcases}:  one that regenerates into two lines and a conic and the other that regenerates into one line and two conics.  We consider each type at a time, demonstrating what orderings of the three lines are possible and showing examples.


Let us consider a 3-point of the first type in a degenerated surface as in Figure \ref{fig:two3ptcases} (I).

\begin{proposition}
Given a 3-point of the first type, any ordering of the three lines $i$, $j$, and $k$ can be obtained by the lexicographic ordering of the vertices in the degenerated surface, as shown in Figure \ref{fig:3ptcasesTypeIordering}.
\end{proposition}

\begin{figure}[h]
\begin{center}
\begin{tikzpicture}

\draw (-1,0.5) node {$(A)$};

\draw (0,1) -- (0,0) -- (1,0);
\draw (0,0) -- (1,1);
\draw (-.25,0.5) node {$k$};
\draw (0.75,-.25) node {$i$};
\draw (0.75,0.5) node {$j$};
\fill[color=black] (0,0) circle (3pt);
\draw (-.25,-.25) node {$P$};
\draw (1.25,-.25) node {$Q$};
\draw (-.25,1.25) node {$R$};
\draw (1.25,1.25) node {$S$};

\draw (2,0.5) node {$(B)$};

\draw[shift={(3,0)}] (0,0) -- (0,1) -- (1,1);
\draw[shift={(3,0)}] (1,0) -- (0,1);
\draw[shift={(3,0)}] (-.25,0.5) node {$i$};
\draw[shift={(3,0)}] (0.75,1.25) node {$k$};
\draw[shift={(3,0)}] (0.75,0.5) node {$j$};
\fill[shift={(3,0)}, color=black] (0,1) circle (3pt);
\draw[shift={(3,0)}] (-.25,-.25) node {$P$};
\draw[shift={(3,0)}] (1.25,-.25) node {$Q$};
\draw[shift={(3,0)}] (-.25,1.25) node {$R$};
\draw[shift={(3,0)}] (1.25,1.25) node {$S$};

\draw (5,0.5) node {$(C)$};

\draw[shift={(6,0)}] (1,0) -- (1,1) -- (0,1);
\draw[shift={(6,0)}] (0,0) -- (1,1);
\draw[shift={(6,0)}] (1.25,0.5) node {$k$};
\draw[shift={(6,0)}] (0.75,1.25) node {$i$};
\draw[shift={(6,0)}] (0.75,0.5) node {$j$};
\fill[shift={(6,0)}, color=black] (1,1) circle (3pt);
\draw[shift={(6,0)}] (-.25,-.25) node {$P$};
\draw[shift={(6,0)}] (1.25,-.25) node {$Q$};
\draw[shift={(6,0)}] (-.25,1.25) node {$R$};
\draw[shift={(6,0)}] (1.25,1.25) node {$S$};

\end{tikzpicture}
    \caption{Possible orderings of edges around a $3$-point of the first type.}
    \label{fig:3ptcasesTypeIordering}
\end{center}
\end{figure}
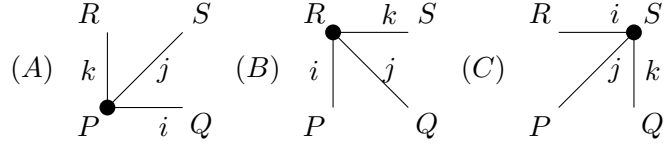

\begin{proof}
Following Figure \ref{fig:3ptcasesTypeIordering}, we assign coordinates to the vertices $P=(x_1,y_1)$, $Q=(x_2,y_1)$, $R=(x_1,y_2)$, and $S=(x_2,y_2)$ with $x_1<x_2$ and $y_1<y_2$.

In Figure \ref{fig:3ptcasesTypeIordering} (A), this gives the ordering of the lines as $i<k<j$, in Figure \ref{fig:3ptcasesTypeIordering} (B), this gives $i<j<k$, and in Figure \ref{fig:3ptcasesTypeIordering} (C) this gives $j<k<i$.
\end{proof}

\begin{example}
Consider the ``$(2,2)$-pillow degeneration'' of a $K3$ surface of degree $16$ which is embedded in $\mathbb{CP}^9$, appearing in \cite{ACMT} and shown in Figure \ref{fig:pillow}.  The boundaries of the two pieces are identified with the left piece on top and right piece on bottom.  Observe that the ordering of the vertices is not purely lexicographic because of identifications of the edges giving non-planarity.  Here the boundary points $v_1$ to $v_8$ are labelled \emph{in this order} using the lexicographic ordering while the interior point on the top is $v_9$ and the bottom is $v_{10}$.  this gives $v_1<\ldots<v_8<v_9<v_{10}$.  The corner vertices $v_1$, $v_3$, $v_6$, and $v_8$ are the $3$-points that appear as Figure \ref{fig:3ptcasesTypeIordering} (A).

\begin{figure}[h]
\begin{center}
\begin{tikzpicture}

\draw (0,0) -- (4,0) -- (4,2) -- (2,0) -- (2,2) -- (0,0) -- (0,2) -- (4,2) -- (4,4) -- (0,4) -- (0,2) -- (2,4) -- (2,2) -- (4,4);

\fill[color=black] (0,0) circle (3pt);
\fill[color=black] (4,0) circle (3pt);
\fill[color=black] (0,4) circle (3pt);
\fill[color=black] (4,4) circle (3pt);

\draw (-.25,-.25) node {$v_6$};
\draw (2,-.25) node {$v_7$};
\draw (4.25,-.25) node {$v_8$};

\draw (-.25,2) node {$v_4$};
\draw (1.75,2.25) node {$v_9$};
\draw (4.25,2) node {$v_5$};

\draw (-.25,4.25) node {$v_1$};
\draw (2,4.25) node {$v_2$};
\draw (4.25,4.25) node {$v_3$};

\draw (10,0) -- (6,0) -- (6,2) -- (8,0) -- (8,2) -- (10,0) -- (10,2) -- (6,2) -- (6,4) -- (10,4) -- (10,2) -- (8,4) -- (8,2) -- (6,4);

\fill[color=black] (6,0) circle (3pt);
\fill[color=black] (10,0) circle (3pt);
\fill[color=black] (6,4) circle (3pt);
\fill[color=black] (10,4) circle (3pt);

\draw (5.75,-.25) node {$v_6$};
\draw (8,-.25) node {$v_7$};
\draw (10.25,-.25) node {$v_8$};

\draw (5.75,2) node {$v_4$};
\draw (8.45,2.25) node {$v_{10}$};
\draw (10.25,2) node {$v_5$};

\draw (5.75,4.25) node {$v_1$};
\draw (8,4.25) node {$v_2$};
\draw (10.25,4.25) node {$v_3$};

\end{tikzpicture}
  \caption{The $(2,2)$-pillow degeneration.}
  \label{fig:pillow}
\end{center}
\end{figure}
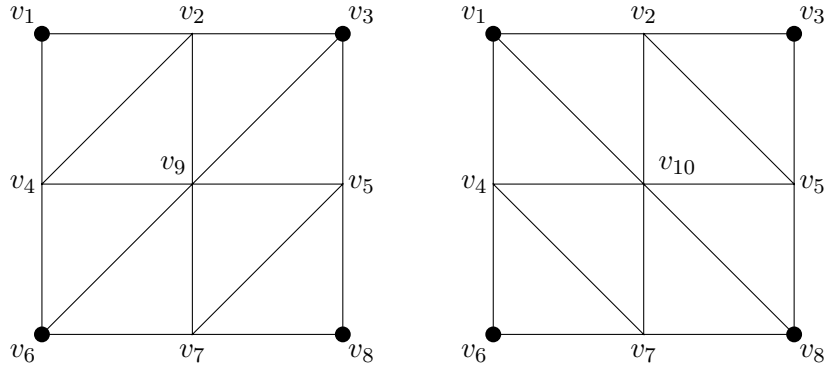

Next we consider another ordering of the same pillow degeneration given in \cite{CMT}.  Here the boundary points $v_1$ to $v_8$ are ordered \emph{clockwise} $v_1<v_2<v_3<v_5<v_8<v_7<v_6<v_4$.  However, the four corner vertices are the same $3$-points that appear in Figure \ref{fig:3ptcasesTypeIordering} (A).

Now consider the complete lexicographic ordering on the first piece that starts from the bottom left corner:  $v_6<v_7<v_8<v_4<v_9<v_5<v_1<v_2<v_3$ followed by $v_{10}$ on the other piece.  Contrasting with the earlier two orderings, only $v_1$ and $v_8$ appear in Figure \ref{fig:3ptcasesTypeIordering} (A) while $v_6$ appears in Figure \ref{fig:3ptcasesTypeIordering} (B) and $v_3$ appears in Figure \ref{fig:3ptcasesTypeIordering} (C)!
\end{example}

%

In Figure \ref{fig:3ptcases} we present the three cases of the first type of $3$-point for three arbitrary indices $i, j, k$, where $i<j<k$.
\begin{figure}[h]
\begin{center}
\begin{tikzpicture}

\draw (-1,0.5) node {$(A)$};

\draw (0,1) -- (0,0) -- (1,0);
\draw (0,0) -- (1,1);
\draw (-.25,0.5) node {$i$};
\draw (0.75,-.25) node {$j$};
\draw (0.75,0.5) node {$k$};
\fill[color=black] (0,0) circle (3pt);

\draw (2,0.5) node {$(B)$};

\draw[shift={(3,0)}] (0,1) -- (0,0) -- (1,0);
\draw[shift={(3,0)}] (0,0) -- (1,1);
\draw[shift={(3,0)}] (-.25,0.5) node {$k$};
\draw[shift={(3,0)}] (0.75,-.25) node {$i$};
\draw[shift={(3,0)}] (0.75,0.5) node {$j$};
\fill[shift={(3,0)}, color=black] (0,0) circle (3pt);

\draw (5,0.5) node {$(C)$};

\draw[shift={(6,0)}] (0,1) -- (0,0) -- (1,0);
\draw[shift={(6,0)}] (0,0) -- (1,1);
\draw[shift={(6,0)}] (-.25,0.5) node {$k$};
\draw[shift={(6,0)}] (0.75,-.25) node {$j$};
\draw[shift={(6,0)}] (0.75,0.5) node {$i$};
\fill[shift={(6,0)}, color=black] (0,0) circle (3pt);

\end{tikzpicture}
    \caption{$3$-points of the first type with $i<j<k$.}
    \label{fig:3ptcases}
\end{center}
\end{figure}
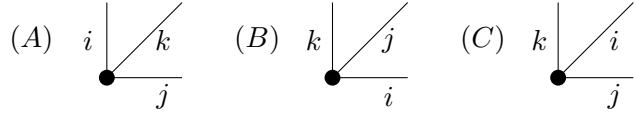

In the regeneration process, the diagonal line is regenerated into a conic which is tangent to both lines. In the next step, each line of the two lines is regenerated into two parallel lines, causing the appearance of three cusps instead of each tangency point. Moreover, each node is regenerated into four  nodes as in Figure \ref{fig:3ptcasesAlgGeom}.  Suppose that two lines intersect at a node, say lines $i$ and $j$.  Recall from Lemma \ref{lem:25} that the resulting braid is $Z^2_{i\; i'\;, j\; j'\;}=Z^2_{i'\; j'}Z^2_{i'\; j}Z^2_{i\; j'}Z^2_{i\; j}$.

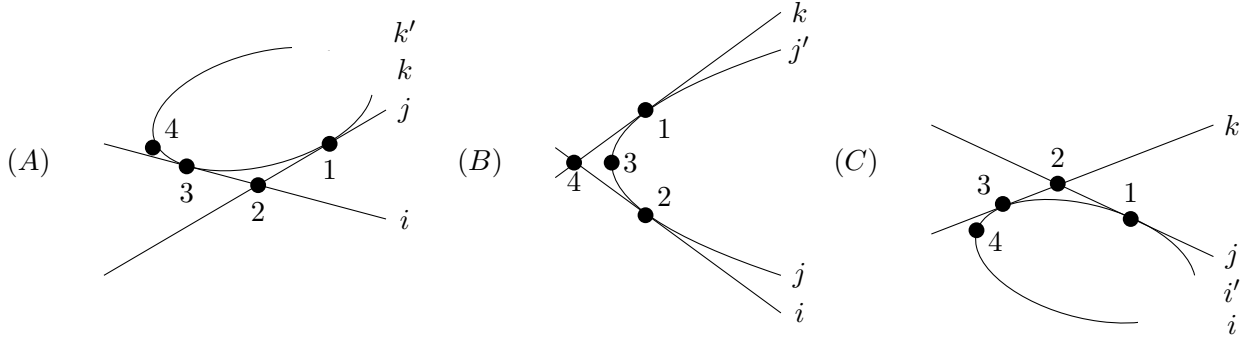
\begin{figure}[h]
\begin{center}
\begin{tikzpicture}

\draw (-7,1.5) node {$(A)$};



\draw[rotate around={15:(-3,3)}] (-3,3) arc (45:345:1.5 and .75);
\draw (-2.25,2.2) -- (-6,0);
\draw (-2.25,.75) --    (-3,.95) -- (-6,1.75);
\fill[color=black] (-5.35,1.7) circle (3pt);
\draw (-5.1,1.95) node {$4$};
\fill[color=black] (-4.9,1.45) circle (3pt);
\draw (-4.9,1.1) node {$3$};
\fill[color=black] (-3.95,1.2) circle (3pt);
\draw (-3.95,.85) node {$2$};
\fill[color=black] (-3,1.75) circle (3pt);
\draw (-3,1.4) node {$1$};

\draw (-2,.75) node {$i$};
\draw (-2,2.2) node {$j$};
\draw (-2,2.75) node {$k$};
\fill[color=white] (-3.5,3) rectangle (-2,4);
\draw (-2,3.25) node {$k'$};

\draw (-1,1.5) node {$(B)$};

\draw (3,0) .. controls (0,1) and (0,2) .. (3,3);
\draw (3,3.5) -- (0,1.3);
\draw (3,-.5) -- (0,1.7);
\fill[color=black] (.25,1.5) circle (3pt);
\draw (.25,1.25) node {$4$};
\fill[color=black] (.75,1.5) circle (3pt);
\draw (1,1.5) node {$3$};
\fill[color=black] (1.2,2.2) circle (3pt);
\draw (1.45,1.95) node {$1$};
\fill[color=black] (1.2,.8) circle (3pt);
\draw (1.45,1.05) node {$2$};

\draw (3.25,-.5) node {$i$};
\draw (3.25,0) node {$j$};
\draw (3.25,3) node {$j'$};
\draw (3.25,3.5) node {$k$};

\draw (4,1.5) node {$(C)$};

\draw[shift={(.5,0)}] [rotate around={-15:(8,0)}] (8,0) arc (15:315:1.5 and .75);
\draw (8.75,2) -- (5,.55);
\draw (5,2) -- (8.75,.25);
\fill[color=black] (5.6,.6) circle (3pt);
\draw (5.85,.45) node {$4$};
\fill[color=black] (5.95,.95) circle (3pt);
\draw (5.7,1.2) node {$3$};
\fill[color=black] (6.68,1.22) circle (3pt);
\draw (6.68,1.57) node {$2$};
\fill[color=black] (7.65,.75)    circle (3pt);
\draw (7.65,1.1) node {$1$};

\fill[color=white] (7.75,-.75) rectangle (8.5,-.5);
\draw (9,-.65) node {$i$};
\draw (9,-.2) node {$i'$};
\draw (9,.25) node {$j$};
\draw (9,2) node {$k$};

\end{tikzpicture}
    \caption{Partial regeneration of the $3$-point cases with $i<j<k$.}
    \label{fig:3ptcasesAlgGeom}
\end{center}
\end{figure}

\begin{proposition}
Given a 3-point of the second type, the only ordering of the three lines obtained by the lexicographic ordering of the vertices is $i\leq j\leq k$ as shown in Figure \ref{fig:3ptcasesTypeIIordering}.
\end{proposition}

\begin{figure}[h]
\begin{center}
\begin{tikzpicture}

\draw (0,1) -- (1,0) -- (2,1);
\draw (1,0) -- (1,1);
\draw (0,0.5) node {$i$};
\draw (2,.5) node {$k$};
\draw (0.85,0.5) node {$j$};
\fill[color=black] (1,0) circle (3pt);
\draw (1,-.5) node {$P$};
\draw (0,1.5) node {$Q$};
\draw (1,1.5) node {$R$};
\draw (2,1.5) node {$S$};

\end{tikzpicture}
    \caption{The only $3$-point case for the second type ordering with $i<j<k$.}
    \label{fig:3ptcasesTypeIIordering}
\end{center}
\end{figure}
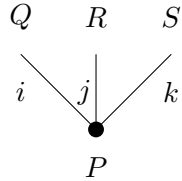

\begin{proof}
Following Figure \ref{fig:3ptcasesTypeIIordering}, we assign coordinates to the vertices $P = (x_2, y_1)$, $Q = (x_1, y_2)$,
$R = (x_2, y_2)$, and $S = (x_3, y_2)$ with $x_1 < x_2 < x_3$ and $y_1 < y_2$.  Following the enumeration of vertices, we conclude that we have the enumeration of lines $i<j<k$.
\end{proof}

\begin{example}
\label{ex:3ptTypeII}
Consider the surface $\mathbb{CP}^1\times\mathbb{CP}^1$.  Take $l_1=\mathbb{CP}^1\times pt$ and $l_2=pt\times\mathbb{CP}^1$.  For $a,b\in\mathbb{N}$, consider the linear combination $al_1+bl_2$.  We embed our surface into a projective space with respect to the linear system $|al_1+bl_2|$.

For this example we take $a=1$ and $b=n$ as in Figure \ref{fig:example:CP1xCP1}.
  Enumerating the vertices in a purely lexicographic manner $v_1<v_2<\ldots<v_{2n+2}$, we obtain the black vertices $v_2,v_4,\ldots,v_{n-1}$ and $v_{n+4},v_{n+6},\ldots$ that are 3-points of the second type.

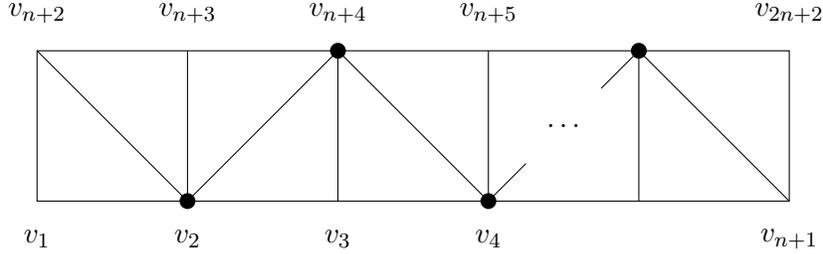
\begin{figure}[h]
\begin{center}
\begin{tikzpicture}

\draw (0,2) -- (0,0) -- (10,0) -- (10,2) -- cycle;
\draw (0,2) -- (2,0) -- (4,2) -- (6,0) -- (6.5,.5);
\draw (7.5,1.5) -- (8,2) -- (10,0);
\draw (2,0) -- (2,2);
\draw (4,0) -- (4,2);
\draw (6,0) -- (6,2);
\draw (8,0) -- (8,2);
\draw (7,1) node {$\ldots$};
\fill[color=black] (2,0) circle (3pt);
\fill[color=black] (4,2) circle (3pt);
\fill[color=black] (6,0) circle (3pt);
\fill[color=black] (8,2) circle (3pt);

\draw (0,-.5) node {$v_1$};
\draw (2,-.5) node {$v_2$};
\draw (4,-.5) node {$v_3$};
\draw (6,-.5) node {$v_4$};
\draw (10,-.5) node {$v_{n+1}$};

\draw (0,2.5) node {$v_{n+2}$};
\draw (2,2.5) node {$v_{n+3}$};
\draw (4,2.5) node {$v_{n+4}$};
\draw (6,2.5) node {$v_{n+5}$};
\draw (10,2.5) node {$v_{2n+2}$};

%
%

\end{tikzpicture}
    \caption{A degeneration of $\mathbb{CP}^1\times\mathbb{CP}^1$ with $a=1$ and $b=n$.}
    \label{fig:example:CP1xCP1}
\end{center}
\end{figure}

To obtain a second example, we extend the degeneration of the bi-embedding (1,2) of $\mathbb{CP}^1\times\mathbb{CP}^1$ to a degeneration of a singular toric surface embedded in $\mathbb{CP}^6$ as in \cite{AO} and shown in Figure \ref{fig:ogata}.  Here the purely lexicographic ordering of vertices $v_1<v_2<\ldots<v_7$ gives a 3-point of the second type -- the black vertex labelled $v_2$.

\begin{figure}[h]
\begin{center}
\begin{tikzpicture}

\draw (2,0) -- (0,2) -- (2,4) -- (4,2) -- (2,0) -- (0,0) -- (0,2) -- (4,2) -- (4,0) -- (2,0) -- (2,4);

\draw (-.25,-.5) node {$v_1$};
\fill[color=black] (2,0) circle (3pt);
\draw (2,-.5) node {$v_2$};
\draw (4.25,-.5) node {$v_3$};

\draw (-.25,2) node {$v_4$};
\draw (1.75,2.25) node {$v_5$};
\draw (4.25,2) node {$v_6$};

\draw (2,4.25) node {$v_7$};

\end{tikzpicture}
  \caption{A degeneration of a singular toric surface embedded in $\mathbb{CP}^6$.}
  \label{fig:ogata}
\end{center}
\end{figure}
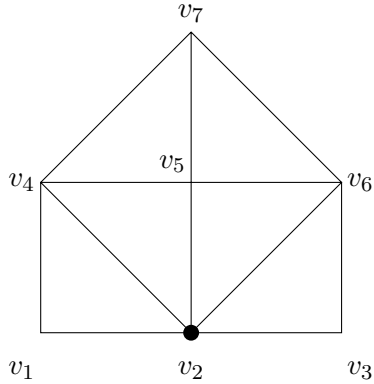

%

\end{example}

\section{Main Results}
\label{sec:Main}

The overall goal of Subsection \ref{subsec:local} is to translate local information from the degeneration of the curve to local information in the resulting braid obtained by the braid monodromy algorithm presented by Moishezon-Teicher \cite{MTII}.  In the first setting we consider the components that meet at each singularity, and in the second setting we consider strands of the braid.  Proposition \ref{prop:takej} states that these global indices are in fact the same across both settings.  Proposition \ref{prop:intact} shows that the order of these singularities does not affect the braiding.

The main result in Subsection \ref{subsec:global} is Proposition \ref{thm:mirror}, which relates a global transformation of the curve to a global transformation of the resulting braid.


These results provide us with the ability to classify the closures of the links obtained from 2-points and 3-points in Section \ref{sec:examples}.

\subsection{Local contributions to the braid monodromy}
\label{subsec:local}

Our first goal is to prove that the indices of the components of the curve meeting at a singularity (in local configurations) tell us where the braiding occurs when considering all strands.


\begin{proposition}
\label{prop:takej}
The global indices $i$ and $j$ of the two components of the curve that meet at a local singularity
are the global indices of the strands
in
the resulting braid for that local singularity.
\end{proposition}

\begin{proof}
Recall that the singularity can be viewed locally as the curve $y^2=x^\varepsilon$ for $\varepsilon=1,2,3,4$ giving, respectively, a branch point, a node, a cusp, and a tangency.  The resulting local braid monodromy for such a singularity is $\phi=Z^\varepsilon_{i\;i+1}$ as in Moishezon-Teicher \cite[p.487 Proposition-Example VI.1.1.]{MTI}\cite[p.487 Proposition-Example VI.1.1.]{MTI}.

This singularity appears in the regeneration of some $k$-point (for $k=2,3$), where we recall that we get only branch points, nodes, and cusps.
This $k$-point contributes $\beta_{2k}$, where $2k$ is the degree of the local regeneration as well as the number of strands in this braid, to the overall braid monodromy factorization, and we embed $\phi\hookrightarrow \beta_{2k}$.

The construction of $\beta_{2k}$ is determined by the regeneration of the embedding of braid groups $B_k\hookrightarrow B_n$ into $B_{2k}\hookrightarrow B_{2n}$, where the index of $B$ is the number of strands of the braid, and recall that $n$ is the number of lines in the projection of the degenerated surface $X_0$ to $S_0$.

Then we have $\phi\hookrightarrow\beta_{2k}\hookrightarrow \Delta^2_{2n}$, where the braid monodromy factorization contributes $\Delta_{2n}^2$ globally.  Thus the end points maintain their indices.
\end{proof}

Next we take this local information obtained from the singularity and show that it is unaffected by other singularities.


\begin{proposition}
\label{prop:intact}
The local contribution to the resulting braiding of the strands indexed by Proposition \ref{prop:takej} is unaffected by the order in which the singularity is considered in the braid monodromy process.
\end{proposition}

\begin{proof}
The locally contributed braid $\beta$ can be written as a product $\beta_1,\ldots,\beta_k$ of braids.  The individual conjugations $\alpha^{-1}\beta_i\alpha$ together give the single conjugation $\alpha^{-1}\beta\alpha$, thus leaving $\beta$ intact.

Any additional strands that come into $\beta$ in $\alpha^{-1}$ leave in $\alpha$ in the opposite way.  Thus in the closure the components related to the additional strands do not become linked with the components coming from $\beta$.  In particular, these strands can themselves be conjugated to return back to the identity.
\end{proof}

In the regeneration process, each indexed component $i$ in the branch curve of $n$ components becomes two components $i$ and $i'$.  In braid notation on $2n$ strands, these correspond to strands $2i-1$ and $2i$, as in the following proposition.

\begin{proposition}
\label{prop:BranchCusp}
The following products of braids correspond to Figure \ref{fig:2ptcases} (A1), (A2), (B1), (B2) respectively, which depict cusps and branch points.
\begin{eqnarray}
\label{LeftCusp1}
(Z_{i'\;i+1}^3)^{Z_{i+1\;i+1'}} \cdot Z^3_{i'\;i+1} \cdot (Z^3_{i'\;i+1})^{Z^{-1}_{i+1\;i+1'}} &=&
(\sigma_{2i+1}^{-1}\sigma_{2i}^3\sigma_{2i+1})\cdot\sigma_{2i}^3\cdot(\sigma_{2i+1}\sigma_{2i}^3\sigma_{2i+1}^{-1}) \\
\label{LeftCusp2}
(Z_{i\; i'})^{Z^2_{i',\; i+1\; i+1'}} &=& (\sigma_{2i}\sigma_{2i+1}^2\sigma_{2i})^{-1}\sigma_{2i-1}(\sigma_{2i}\sigma_{2i+1}^2\sigma_{2i}) \\
\label{RightCusp1}
(Z_{i\;i+1}^3)^{Z_{i\;i'}} \cdot Z^3_{i\;i+1} \cdot (Z^3_{i\; i+1})^{Z^{-1}_{i\;i'}} &=& \sigma_{2i}^3\cdot(\sigma_{2i}^{-1}\sigma_{2i-1}^3\sigma_{2i})\cdot(\sigma_{2i-1}^2\sigma_{2i}^3\sigma_{2i-1}^{-2}) \\
\label{RightCusp2}
(Z_{i+1\; i+1'})^{Z^2_{i+1,\; i\; i'}} &=& (\sigma_{2i}\sigma_{2i-1}^2\sigma_{2i})^{-1}\sigma_{2i+1}(\sigma_{2i}\sigma_{2i-1}^2\sigma_{2i})
\end{eqnarray}
\end{proposition}

\begin{proof}
According to the regeneration rule for tangency (Theorem \ref{3rule}) we get the left hand sides of the equations.  Using Property \ref{property:node} 
we get the right hand sides.

The product of the three braids in  \ref{LeftCusp1} and \ref{RightCusp1} corresponds to the three cusps in Figure \ref{fig:2ptcases} (A1) and (B1) respectively.  These can be simplified to $\sigma_{2i}\sigma_{2i+1}^3\sigma_{2i}\sigma_{2i+1}^3\sigma_{2i}$ and $\sigma_{2i}\sigma_{2i-1}^3\sigma_{2i}\sigma_{2i-1}^3\sigma_{2i}$ respectively.  Moreover, we get the braids in \ref{LeftCusp2} and \ref{RightCusp2} corresponding to the branch points in Figure \ref{fig:2ptcases} (A2) and (B2) respectively.
\end{proof}

To demonstrate these last few results, consider the following example.

\begin{example}
The braid $(Z_{1\;2})^{Z^2_{2\;4}Z^2_{2\;3}}$ is given in Figure \ref{fig:exampleconj}.

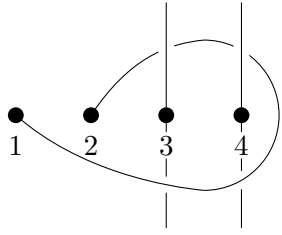
\begin{figure}[h]
\begin{center}
\begin{tikzpicture}

\draw (9,0) -- (9,-1.5);
\draw (10,0) -- (10,-1.5);

\draw[color=white, line width=6pt] (9.5,-1) arc (-90:90:1);
\draw (9.5,-1) arc (-90:90:1);
\draw (8,0) .. controls (8.5,.75) and (9,.95) .. (9.5,1);
\draw[color=white, line width=6pt] (7,0) .. controls (7.85,-.75) and (9,-.95) .. (9.5,-1);
\draw (7,0) .. controls (7.85,-.75) and (9,-.95) .. (9.5,-1);

\draw[color=white, line width=6pt] (9,0) -- (9,1);
\draw (9,0) -- (9,1.5);

\draw[color=white, line width=6pt] (10,0) -- (10,1);
\draw (10,0) -- (10,1.5);

\fill[color=black] (7,0) circle (3pt);
\fill[color=black] (8,0) circle (3pt);
\fill[color=black] (9,0) circle (3pt);
\fill[color=black] (10,0) circle (3pt);
\draw (7,-0.4) node {$1$};
\draw (8,-0.4) node {$2$};
\fill[color=white] (9,-.4) circle (5pt);
\draw (9,-0.4) node {$3$};
\fill[color=white] (10,-.4) circle (5pt);
\draw (10,-0.4) node {$4$};

\end{tikzpicture}
    \caption{The braid given by $(Z_{1\;2})^{Z^2_{2\;4}Z^2_{2\;3}}$.}
    \label{fig:exampleconj}
\end{center}
\end{figure}

The resulting braid can be written $(\sigma_2^{-1}\sigma_3^{-2}\sigma_2^{-1})\sigma_1(\sigma_2\sigma_3^2\sigma_2)$.  It is clear here that the components indexed by 3 and 4 are unlinked and can be removed by conjugation.
\end{example}

This leads us to the following results on the number of components of the links we obtain.  In the proposition above we saw indexed components $i$ and $i'$ in the branch curve.  However in several cases these components become a single component in the closure of the related braid.  In order to better see this, we present first the following case analysis for branch points, cusps, and nodes.


\begin{lemma}
\label{lem:branchcomponent}
A conic indexed by $i$ and $i'$ with a branch point contributes a single component to the link obtained by closing the braid.
\end{lemma}


\begin{lemma}
\label{lem:tangentcomponent}
A line indexed by $j$ and $j'$ (respectively $i$ and $i'$) tangent to a conic at a singularity that regenerates into three cusps contributes a single component to the link obtained by closing the braid; together with the conic $i$ and $i'$ (respectively $j$ and $j'$) this gives two components.
\end{lemma}




Recall the regeneration of a node as in Figure \ref{fig:noderegeneration}.

\begin{lemma}
\label{lem:nodecomponents}
A single node that regenerates into four nodes contributes four components to the link obtained by closing the braid.
\end{lemma}




\subsection{Global contributions to the braid monodromy}
\label{subsec:global}

Next we present a statement that will allow us to move between similar cases.  Let $S$ be an arrangement in $\mathbb{CP}^2$.  It can be a plane curve or a branch curve of a surface $X$ or of one of its degenerations.  Let $r(S)$ be the rotation of $180^\circ$ of  $S$ around the horizontal line that is the chosen axis for the ordering of the singularities.

It is natural that the braid monodromy factorization of $r(S)$ is a rotation of the braid monodromy factorization of $S$.  Nevertheless it is not written anywhere, and for the sake of clarification for the reader and for coherence we include a proof here.

\begin{proposition}
\label{thm:mirror}
The local contributions to the braid monodromy factorization obtained from $S$ and $r(S)$ are related by a rotation around a line parallel to the direction of the braid.  Moreover the closures of these braids give the same links.
\end{proposition}

For the sake of clarification, we take the braid to be depicted vertically and thus will consider rotation of the braid around a vertical line.  However, some of the figures below depict the braids horizontally to conserve space.

\begin{proof}
First observe that the order of the singularities in $S$ is preserved by the rotation.  So we need only consider a single singularity.  We start by considering only a partial regeneration in which we have a branch point, simple node, or tangency.  Afterwards we continue to a complete regeneration where a simple node is regenerated into four nodes and tangency is regenerated into three cusps.

Suppose components $i$ and $j$ meet at a singularity.  
  The resulting braid from the braid monodromy algorithm is a conjugation of $Z^\varepsilon_{i\; j}$, where $\varepsilon$ is the exponent associated to the type of singularity as in \cite{MTII}.  By Property \ref{property:node}, $Z_{i\;j}$ can be rewritten as
\begin{equation}
\label{eq:ex:rotationBEFORE}
Z_{i\;j}=(\sigma_i\ldots\sigma_{j-2})\sigma_{j-1}(\sigma_i\ldots\sigma_{j-2})^{-1}=(\sigma_{i+1}\ldots\sigma_{j-1})^{-1}\sigma_{i}(\sigma_{i+1}\ldots\sigma_{j-1}).
\end{equation}

In $r(S)$ the components that meet at the associated singularity are numbered $(n+1)-i$ and $(n+1)-j$.

Now let us consider the local contribution to the braid monodromy obtained from this singularity in $r(S)$.  This braid is a conjugation of some $Z^\varepsilon_{(n+1)-j\;(n+1)-i}$, where $\varepsilon$ is the exponent associated to the type of singularity as in \cite{MTII}.  By Property \ref{property:node}, this can be re-written as
\begin{eqnarray}
\label{eq:ex:rotationAFTER}
Z_{(n+1)-j\;(n+1)-i} & = & (\sigma_{(n+1)-j}\ldots\sigma_{(n+1)-i-2})\sigma_{(n+1)-i-1}(\sigma_{(n+1)-j}\ldots\sigma_{(n+1)-i-2})^{-1}\\
& = & (\sigma_{(n+1)-j+1}\ldots\sigma_{(n+1)-i-1})^{-1}\sigma_{(n+1)-j}(\sigma_{(n+1)-j+1}\ldots\sigma_{(n+1)-i-1}).\nonumber
\end{eqnarray}

A rotation of some generator $\sigma_k^{\pm1}$ around a vertical line is $\sigma_{n-k}^{\pm1}$.  Thus we can see that Equations \ref{eq:ex:rotationBEFORE} and \ref{eq:ex:rotationAFTER} are related by a reflection around a vertical line.  In this way this rotation also holds for any conjugation, as this conjugation can be expressed in terms of the $\sigma_i$'s.

Having proved the partial regeneration, we move to the complete regeneration.  For the case of a branch point, we apply Proposition \ref{prop:takej} to the regeneration where some point $i$ becomes two points $i$ and $i'$.  The next two lemmas consider the cases of a single node and a tangency.

\begin{lemma}
\label{lem:fournodes}
For the case of a single node regenerated into four nodes in $S$, as is locally depicted in Figure \ref{fig:noderegeneration}, consider the resulting braid from this regeneration and the resulting braid from the regeneration of the singularity in $r(S)$.  Then these braids are related by a rotation around a vertical line.
\end{lemma}

\begin{proof}
Suppose that in $S$ the indices of the lines are $i$ and $j$; by Proposition \ref{prop:takej} we may take $j=i+1$.  Then the resulting braid is $Z^2_{i\;i'\;j\;j'}$, which by Proposition \ref{lem:25} and Property \ref{property:node} gives
\begin{equation}
Z^2_{i'\;j'}Z^2_{i'\;j}Z^2_{i\;j'}Z^2_{i\;j}=(\sigma^{-1}_{2i+1}\sigma^2_{2i}\sigma_{2i+1})(\sigma^2_{2i})(\sigma^{-1}_{2i+1}\sigma^{-1}_{2i}\sigma^2_{2i-1}\sigma_{2i}\sigma_{2i+1})(\sigma^{-1}_{2i}\sigma^2_{2i-1}\sigma_{2i}).
\end{equation}
This simplifies to $\sigma_{2i}\sigma_{2i-1}\sigma_{2i+1}\sigma^2_{2i}\sigma_{2i-1}\sigma_{2i+1}\sigma_{2i}$.

In the rotation $r(S)$ the indices of the lines are $(n+1)-j=(n+1)-(i+1)=(n-i)$ and $(n+1)-i=(n-i+1)$.  Then the resulting braid is $Z^2_{(n-i)\;(n-i)'\;(n-i+1)\;(n-i+1)'}$, which by Proposition \ref{lem:25} and Property \ref{property:node} simplifies to $\sigma_{2(n-i)}\sigma_{2(n-i)+1}\sigma_{2(n-i)-1}\sigma^2_{2(n-i)}\sigma_{2(n-i)+1}\sigma_{2(n-i)-1}\sigma_{2(n-i)}$.

Then it is clear that these two simplifications are related by a rotation around a vertical line with $\sigma_{2i}$ rotated to $\sigma_{2n-(2i)}=\sigma_{2(n-i)}$, $\sigma_{2i-1}$ rotated to $\sigma_{2n-(2i-1)}=\sigma_{2(n-i)+1}$, and $\sigma_{2i+1}$ rotated to $\sigma_{2n-(2i+1)}=\sigma_{2(n-i)-1}$.
\end{proof}

\begin{lemma}
\label{lem:threecusps}
For the case of a tangency regenerated into three cusps in $S$, as is locally depicted in Figure \ref{fig:2ptcasesAlgGeom}, consider the resulting braid from this regeneration and the resulting braid from the regeneration of the singularity in $r(S)$.  Then these braids are related by a rotation around a vertical line.
\end{lemma}

\begin{proof}
Suppose that in $S$ the index $i$ of the conic is less than the index $j$ of the line as in Figure \ref{fig:2ptcasesAlgGeom} (A).  By Proposition \ref{prop:takej} we may take $j=i+1$.  Then the resulting braid is $Z^3_{i'\;j\;j'}$, which by Proposition \ref{prop:BranchCusp} can be written as in Equation \ref{LeftCusp1} using the regeneration rule for tangency (Theorem \ref{3rule}).  This simplifies to $\sigma_{2i}\sigma_{2i+1}^3\sigma_{2i}\sigma_{2i+1}^3\sigma_{2i}$.

In the rotation $r(S)$ the index $(n+1)-j=(n+1)-(i+1)=(n-i)$ of the line is less than the index $(n+1)-i=(n-i+1)$ of the conic as in Figure \ref{fig:2ptcasesAlgGeom} (B).  Then the resulting braid is $Z^3_{(n-i)\;(n-i)'\;(n-i+1)}$, which by Proposition \ref{prop:BranchCusp} can be written as in Equation \ref{RightCusp1} using the regeneration rule for tangency (Theorem \ref{3rule}).  This simplifies to $\sigma_{2(n-i)}\sigma_{2(n-i)-1}^3\sigma_{2(n-i)}\sigma_{2(n-i)-1}^3\sigma_{2(n-i)}$.

Then it is clear that Equations \ref{LeftCusp1} and \ref{RightCusp1} are related by a rotation around a vertical line with $\sigma_{2i}$ rotated to $\sigma_{2n-(2i)}=\sigma_{2(n-i)}$ and $\sigma_{2i+1}$ rotated to $\sigma_{2n-(2i+1)}=\sigma_{2(n-i)-1}$.
\end{proof}

Finally if the braids are related by a rotation around a vertical line, then the closures of these braids are the same links.
\end{proof}

\section{Examples}
\label{sec:examples}

In this section the three main examples we will consider will come from 2-points and 3-points.  We exhibit their braid monodromies and give some properties of their closures as links.

\begin{ex-prop}
\label{prop:2ptexample}
The closure of the braid for the complete regeneration for a 2-point is a two-component link where each component is itself unknotted and the two components have linking number four.
\end{ex-prop}

\begin{proof}
By Proposition \ref{prop:BranchCusp}, the braid for the complete regeneration for a 2-point is
\begin{equation}
\label{eq:2ptA}
\sigma_2\sigma_3^2\sigma_2(\sigma_1\sigma_3)\sigma_2\sigma_3^2\sigma_2
\end{equation}
for the partial regeneration given in Figure \ref{fig:2ptcasesAlgGeom} (A) and
\begin{equation}
\label{eq:2ptB}
\sigma_2\sigma_1^2\sigma_2(\sigma_1\sigma_3)\sigma_2\sigma_1^2\sigma_2
\end{equation}
for the partial regeneration given in Figure \ref{fig:2ptcasesAlgGeom} (B).  We depict these braids horizontally in Figure \ref{fig:single2ptbraid}.

\begin{figure}[h]
\begin{center}
\begin{tikzpicture}

\draw (-1,1.5) node {$(A)$};

\foreach \x/ \y in {0/1, 1/0, 2/0, 3/1, 4/0, 4/2, 5/1, 6/0, 7/0, 8/1}
    {
    \draw (\x+1,\y) -- (\x,\y+1);
    \draw[color=white, line width=10] (\x,\y) -- (\x+1,\y+1);
    \draw (\x,\y) -- (\x+1,\y+1);
    }

\draw (0,0) -- (1,0);
\draw (1,2) -- (3,2);
\draw (3,0) -- (4,0);
\draw (5,0) -- (6,0);
\draw (6,2) -- (8,2);
\draw (8,0) -- (9,0);
\draw (0,3) -- (4,3);
\draw (5,3) -- (9,3);

\draw (-1,-2.5) node {$(B)$};

\foreach \x/ \y in {0/-3, 1/-2, 2/-2, 3/-3, 4/-4, 4/-2, 5/-3, 6/-2, 7/-2, 8/-3}
    {
    \draw (\x+1,\y) -- (\x,\y+1);
    \draw[color=white, line width=10] (\x,\y) -- (\x+1,\y+1);
    \draw (\x,\y) -- (\x+1,\y+1);
    }

\draw (0,-1) -- (1,-1);
\draw (1,-3) -- (3,-3);
\draw (3,-1) -- (4,-1);
\draw (5,-1) -- (6,-1);
\draw (6,-3) -- (8,-3);
\draw (8,-1) -- (9,-1);
\draw (0,-4) -- (4,-4);
\draw (5,-4) -- (9,-4);

\end{tikzpicture}
    \caption{The braids for the complete regeneration for a 2-point given in Equations (A) \ref{eq:2ptA} and (B) \ref{eq:2ptB}.}
    \label{fig:single2ptbraid}
\end{center}
\end{figure}
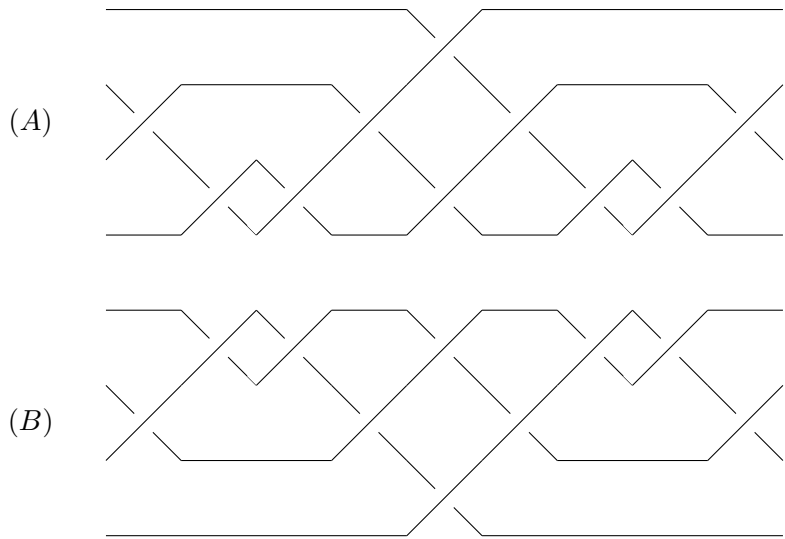

These braids are related by rotation around a line parallel to the direction of the braid (horizontally in the figure), and thus the closures give the same link by Proposition \ref{thm:mirror}.  Thus we only consider Figure \ref{fig:2ptcasesAlgGeom} (A) where the conic is labelled by $i$ and $i'$ and the line is labelled by $j$.  However, this is only a partial regeneration; in the complete regeneration the line becomes two lines $j$ and $j'$.

Via Proposition \ref{prop:takej} we may associate these four components to strands in the braid.

By Lemma \ref{lem:branchcomponent}, the branch point of the conic contributes a single component to the link because the two strands labelled by $i$ and $i'$ are twisted by $Z_{i\;i'}$.  
See Figure \ref{fig:new2ptbraidsB} for one of these braids depicted horizontally.

\begin{figure}[h]
\begin{center}
\begin{tikzpicture}


\foreach \x/ \y in {4/-2, 5/-3, 6/-4, 7/-4, 8/-3}
    {
    \draw (\x+1,\y) -- (\x,\y+1);
    \draw[color=white, line width=10] (\x,\y) -- (\x+1,\y+1);
    \draw (\x,\y) -- (\x+1,\y+1);
    }

\foreach \x/ \y in {0/-3,1/-4,2/-4,3/-3}
    {
    \draw (\x,\y) -- (\x+1,\y+1);
    \draw[color=white, line width=10] (\x+1,\y) -- (\x,\y+1);
    \draw (\x+1,\y) -- (\x,\y+1);
    }

\draw (0,-4) -- (1,-4);
\draw (1,-2) -- (3,-2);
\draw (3,-4) -- (6,-4);
\draw (6,-2) -- (8,-2);
\draw (8,-4) -- (9,-4);
\draw (0,-1) -- (4,-1);
\draw (4,-3) -- (5,-3);
\draw (5,-1) -- (9,-1);

\end{tikzpicture}
    \caption{One of the braids obtained from the regeneration of a branch point as in Lemma \ref{lem:branchcomponent} and Figure \ref{fig:2ptcases} (A2).}
    \label{fig:new2ptbraidsB}
\end{center}
\end{figure}
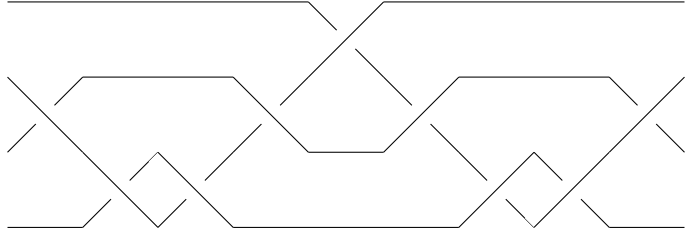

By Lemma \ref{lem:tangentcomponent}, the tangency twists the two strands $j$ and $j'$ with $Z_{j\;j'}$, creating a single component for the line.  Furthermore, the tangency links the $i'$-th strand with both strands $j$ and $j'$ giving a linking number of four total.  See Figure \ref{fig:new2ptbraidsA} for one of these braids depicted horizontally.

\begin{figure}[h]
\begin{center}
\begin{tikzpicture}


\foreach \x/ \y in {0/2, 1/1, 2/1, 3/1, 4/2, 5/1, 6/1, 7/1, 8/2}
    {
    \draw (\x+1,\y) -- (\x,\y+1);
    \draw[color=white, line width=10] (\x,\y) -- (\x+1,\y+1);
    \draw (\x,\y) -- (\x+1,\y+1);
    }


\draw (0,4) -- (9,4);
\draw (0,1) -- (1,1);
\draw (1,3) -- (4,3);
\draw (4,1) -- (5,1);
\draw (5,3) -- (8,3);
\draw (8,1) -- (9,1);


\end{tikzpicture}
    \caption{One of the braids obtained from the regeneration of a tangency as in Lemma \ref{lem:tangentcomponent} and Figure \ref{fig:2ptcases} (A1).}
    \label{fig:new2ptbraidsA}
\end{center}
\end{figure}
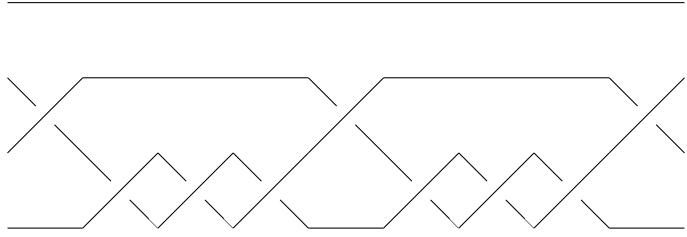

The link may be stabilized to get rid of the first strand resulting in $\sigma_1\sigma_2^2\sigma_1\sigma_2\sigma_1\sigma_2^2\sigma_1$ and its rotation.

In particular, this link is L4a1 where one component is replaced by its (2,1)-cable.
\end{proof}

\begin{remark}
Recall Proposition 16 (A) above in which we considered the partial regeneration for a line tangent to a conic.  There we obtained the link L4a1 without the cabling.
\end{remark}

\begin{ex-prop}
\label{prop:3ptAexample}
The closure of the braid for the complete regeneration for a 3-point of the first type is a three-component link where each component is itself unknotted and each pair of components has linking number four.
\end{ex-prop}

\begin{proof}
Recall that for a 3-point of the first type there are three cases as in Figure \ref{fig:3ptcasesAlgGeom}, where the singularities are ordered.

By Proposition \ref{prop:takej} we may take $(i,j,k)=(1,2,3)$.

Then in the first case (A), the four singularities give, respectively, local contributions as follows:
\begin{enumerate}
	\item $\sigma_4^3(\sigma_4^{-1}\sigma_3^3\sigma_4)(\sigma_3^2\sigma_4^3\sigma_3^{-2})$
	\item $(\sigma_4^{-2}\sigma_3^{-1}\sigma_2^2\sigma_3\sigma_4^2)(\sigma_3\sigma_4^{-2}\sigma_3^{-1}\sigma_2^2\sigma_3\sigma_4^2\sigma_3^{-1})(\sigma_4^{-2}\sigma_3^{-1}\sigma_2^{-1}\sigma_1^2\sigma_2\sigma_3\sigma_4^2)(\sigma_3\sigma_4^{-2}\sigma_3^{-1}\sigma_2^{-1}\sigma_1^2\sigma_2\sigma_3\sigma_4^2\sigma_3^{-1})$
	\item $(\sigma_4^{-1}\sigma_3^{-1}\sigma_2^3\sigma_3\sigma_4)(\sigma_4^{-1}\sigma_3^{-1}\sigma_2^{-1}\sigma_1^3\sigma_2\sigma_3\sigma_4)(\sigma_1^2\sigma_4^{-1}\sigma_3^{-1}\sigma_2^3\sigma_3\sigma_4\sigma_1^{-2})$
	\item $\sigma_4^{-1}\sigma_3^{-1}\sigma_2^{-1}\sigma_1^{-2}\sigma_2^{-1}\sigma_3^{-1}\sigma_4^{-1}\sigma_5\sigma_4\sigma_3\sigma_2\sigma_1^2\sigma_2\sigma_3\sigma_4$
\end{enumerate}

In the second case (B), we have:
\begin{enumerate}
	\item $\sigma_2^3(\sigma_2^{-1}\sigma_1^3\sigma_2)(\sigma_1^2\sigma_2^3\sigma_1^{-2})$
	\item $(\sigma_5^{-1}\sigma_4^3\sigma_5)\sigma_4^3(\sigma_5\sigma_4^3\sigma_5^{-1})$
	\item $\sigma_4^{-1}\sigma_5^{-2}\sigma_4^{-1}\sigma_2^{-1}\sigma_1^{-2}\sigma_2^{-1}\sigma_3\sigma_2\sigma_1^2\sigma_2\sigma_4\sigma_5^2\sigma_4$
	\item $(\sigma_5^{-1}\sigma_4^{-1}\sigma_3\sigma_2^2\sigma_3^{-1}\sigma_4\sigma_5)(\sigma_4^{-1}\sigma_3\sigma_2^2\sigma_3^{-1}\sigma_4)(\sigma_5^{-1}\sigma_4^{-1}\sigma_3\sigma_2^{-1}\sigma_1^2\sigma_2\sigma_3^{-1}\sigma_4\sigma_5)(\sigma_4^{-1}\sigma_3\sigma_2^{-1}\sigma_1^2\sigma_2\sigma_3^{-1}\sigma_4)$
\end{enumerate}

And in the third case (C), we have:
\begin{enumerate}
	\item $(\sigma_3^{-1}\sigma_2^3\sigma_3)\sigma_2^3(\sigma_3\sigma_2^3\sigma_3^{-1})$
	\item $(\sigma_5^{-1}\sigma_3\sigma_2^{-2}\sigma_3^{-1}\sigma_4^2\sigma_3\sigma_2^2\sigma_3^{-1}\sigma_5)(\sigma_3\sigma_2^{-2}\sigma_3^{-1}\sigma_4^2\sigma_3\sigma_2^2\sigma_3^{-1})(\sigma_5^{-1}\sigma_3^2\sigma_2^{-2}\sigma_3^{-1}\sigma_4^2\sigma_3\sigma_2^2\sigma_3^{-2}\sigma_5)(\sigma_3^2\sigma_2^{-2}\sigma_3^{-1}\sigma_4^2\sigma_3\sigma_2^2\sigma_3^{-2})$
	\item $(\sigma_5^{-1}\sigma_4\sigma_3\sigma_2^3\sigma_3^{-1}\sigma_4^{-1}\sigma_5)(\sigma_4\sigma_3\sigma_2^3\sigma_3^{-1}\sigma_4^{-1})(\sigma_5\sigma_4\sigma_3\sigma_2^3\sigma_3^{-1}\sigma_4^{-1}\sigma_5^{-1})$
	\item $\sigma_2^{-1}\sigma_3^{-1}\sigma_4^{-1}\sigma_5^{-2}\sigma_4^{-1}\sigma_3^{-1}\sigma_2^{-1}\sigma_1\sigma_2\sigma_3\sigma_4\sigma_5^2\sigma_4\sigma_3\sigma_2$
\end{enumerate}


The simplification for (A) is a word of length 27 with all its generators positive: $$\sigma_4\sigma_3^2\sigma_4\sigma_2\sigma_3\sigma_1\sigma_2\sigma_4\sigma_3^2\sigma_2(\sigma_4\sigma_3\sigma_2\sigma_1)\sigma_2\sigma_3(\sigma_5\sigma_4\sigma_3\sigma_2\sigma_1)(\sigma_1\sigma_2\sigma_3\sigma_4).$$

The simplification for (B) is also a word of length 27 with all its generators positive: $$(\sigma_4\sigma_3\sigma_2\sigma_1)(\sigma_5\sigma_4\sigma_3\sigma_2)(\sigma_1\sigma_2\sigma_3\sigma_4\sigma_5)(\sigma_1\sigma_2\sigma_3\sigma_4)(\sigma_3\sigma_2\sigma_1\sigma_3\sigma_2\sigma_3\sigma_2\sigma_1\sigma_3\sigma_2).$$

The braid in (C) is related by rotation around a line parallel with the braid to the braid in (A), and thus the closures give the same link by Proposition \ref{thm:mirror}.

Via Proposition \ref{prop:takej} we may associate the six components to six strands in the braid.  By Lemma \ref{lem:branchcomponent}, the branch point of the conic contributes a single component to the link.  By Lemma \ref{lem:tangentcomponent}, each tangency contributes a single component for the relevant line and links the conic with this line giving a linking number of four total.  By Lemma \ref{lem:nodecomponents}, the node on the two lines links the two line components giving a linking number of four total.

The closures of (A) and (B) are in fact the same link; this can be given by the torus link $T(3,3)$ 
 with each component replaced by its $(2,1)$-cable.
%
\end{proof}

\begin{ex-prop}
\label{prop:3ptBexample}
The closure of the braid for the complete regeneration for a 3-point of the second type is a four-component link where each component is itself unknotted, whose linking numbers are given by the following matrix:
\begin{equation} 
\left( 
\begin{matrix} 
  & 1 & 2 & 2\\ 
1 &   & 2 & 2\\ 
2 & 2 &   & 4\\ 
2 & 2 & 4 &   
\end{matrix} 
\right).
\end{equation}
\end{ex-prop}

\begin{proof}
Recall that for a 3-point of the second type there is just a single case as in Figure \ref{fig:fig15}.  The first two singularities give three cusps and a branch point, the next gives four nodes, and the last two give another three cusps and a branch point.

By Proposition \ref{prop:takej} we may take $(i,j,k)=(1,2,3)$.


Then the singularities give, respectively, local contributions as follows:
\begin{enumerate}
	\item $\sigma_4^3(\sigma_4^{-1}\sigma_3^3\sigma_4)(\sigma_3^2\sigma_4^3\sigma_3^{-2})$
	\item $\sigma_4^{-1}\sigma_3^{-2}\sigma_4^{-1}\sigma_5\sigma_4\sigma_3^2\sigma_4$
	\item $(\sigma_5^{-1}\sigma_4^{-1}\sigma_3^{-1}\sigma_2^2\sigma_3\sigma_4\sigma_5)(\sigma_5^{-1}\sigma_4^{-1}\sigma_3^{-1}\sigma_2^{-1}\sigma_1^2\sigma_2\sigma_3\sigma_4\sigma_5)\\
(\sigma_4^{-1}\sigma_3^{-2}\sigma_4^{-2}\sigma_3^{-1}\sigma_2^2\sigma_3\sigma_4^{2}\sigma_3^2\sigma_4)(\sigma_4^{-1}\sigma_3^{-2}\sigma_4^{-2}\sigma_3^{-1}\sigma_2^{-1}\sigma_1^2\sigma_2\sigma_3\sigma_4^2\sigma_3^2\sigma_4)$
	\item $(\sigma_3^{-1}\sigma_2^3\sigma_3)\sigma_2^{3}(\sigma_3\sigma_2^3\sigma_3^{-1})$
	\item $\sigma_2^{-1}\sigma_3^{-2}\sigma_2^{-1}\sigma_1\sigma_2\sigma_3^2\sigma_2$
\end{enumerate}

The simplification is a word of length 28 with all its generators positive: $$\sigma_1\sigma_2\sigma_5\sigma_4\sigma_3\sigma_5\sigma_2\sigma_3\sigma_1\sigma_2(\sigma_2\sigma_3\sigma_4\sigma_5)\sigma_5\sigma_4^2\sigma_5\sigma_3\sigma_4^2\sigma_3\sigma_4\sigma_2\sigma_3\sigma_4^2\sigma_3.$$

Via Proposition \ref{prop:takej} we may associate the six components to six strands in the braid.  By Lemma \ref{lem:branchcomponent}, the branch points of the two conics each contributes a single component to the link.  By Lemma \ref{lem:tangentcomponent}, each tangency contributes a single half-twist on the two components associated to the line, giving a single full-twist of the two strands associated to the line, linking each of these components with each of the conic components twice.  This accounts for singularities (1), (2), (4), and (5) on the list above.


The remaining singularity does not actually appear in the partial regeneration shown in Figure \ref{fig:fig15}.  This is a node at which the conics intersect, and we apply the regeneration rule for a node Theorem\ref{2rule}.  By Lemma \ref{lem:nodecomponents} this gives links the two conic components with linking number four.

This link is the torus link $T(3,3)$ with two strands (associated with the conics) replaced by their $(2,1)$-cables and one strand (associated with the line) replaced by its $(2,2)$-cable.
\end{proof}

These last two examples lead to the obvious conclusion.

\begin{corollary}
The 3-points of the first type and the 3-points of the second type give different braids.
\end{corollary}

\section*{Acknowledgments} This work was supported in part by the Edmund Landau
Center for Research in Mathematics,  the Emmy Noether Research
Institute for Mathematics, the Minerva Foundation (Germany), the
EU-network HPRN-CT-2009-00099(EAGER), the Israel Science
Foundation grant $\#$ 8008/02-3 (Excellency Center ``Group
Theoretic Methods in the Study of Algebraic Varieties''), and  the
Oswald Veblen Fund.

\bibliographystyle{abbrv}
\bibliography{12NovBibliography}

\end{document}